\documentclass[reqno]{amsart}

\usepackage{amsmath,amssymb,graphicx
}
\usepackage{enumitem}

\usepackage{youngtab}
\usepackage[latin1]{inputenc}
\usepackage[colorlinks,citecolor=blue]{hyperref}
\newtheorem{Theorem}{Theorem}[section]
\newtheorem{proposition}[Theorem]{Proposition}

\newtheorem{lemma}[Theorem]{Lemma}
\theoremstyle{definition}
\newtheorem{definition}[Theorem]{Definition}
\newtheorem{Question}[Theorem]{Question}
\newtheorem{example}[Theorem]{Example}

\theoremstyle{remark}
\newtheorem{Remark}[Theorem]{Remark}

\newcommand{\Sh} {{\rm Sh}}

\newcommand{\Typ} {{\rm Type}}

\newcommand{\Red} {{\rm Red}}
\newcommand{\Inv} {{\rm Inv}}
\newcommand{\Tab} {{\rm Tab}}

\newcommand{\A} {\mathcal{A}}
\newcommand{\B} {\mathcal{B}}
\newcommand{\T}{\mathcal{T}}

\newcommand{\SSF}{{\rm SSF}}

\newcommand{\Bal}{{\rm Bal}}

\allowdisplaybreaks

\title{A natural generalization of balanced tableaux}
\author{Fran\c{c}ois Viard}
\date{\today}

\begin{document}
\maketitle

\begin{abstract}
We introduce the notion of ``type'' of a tableau, that allows us to define new families of tableaux including both balanced and standard Young tableaux. We use these new objects to describe the set of reduced decompositions of any permutation. We then generalize the work of Fomin \emph{et al.} by giving, among other things, a new proof of the fact that balanced and standard tableaux are equinumerous, and by exhibiting many new families of tableaux having similar combinatorial properties to those of balanced tableaux.
\end{abstract}

\section{Introduction}

Since the fundamental paper of Stanley \cite{S1}, it is well-known that standard Young tableaux interact with the weak order, providing and important tool to enumerate the reduced expressions of any permutation (see also \cite{EG,FGR,LS,Little}). Another important family of tableaux called \emph{balanced tableaux} has been defined in \cite{EG} by Edelman and Greene to study this enumerative problem. The surprising feature of balanced tableaux is that they are equinumerous with standard tableaux of the same shape $\lambda$. This was showed in the original paper \cite{EG} with a quite involved proof.
The notion of balanced tableaux was further generalized in \cite{FGR}, with the introduction of the set of balanced tableaux of shape $D(\sigma)$, where $D(\sigma)$ denotes the \emph{Rothe diagram} of the permutation $\sigma$. By using this new family, the authors gave in \cite{FGR} a new interpretation of the set of reduced expressions of any permutation, and derived from this a new and more accessible proof of the fact that balanced and standard tableaux are equinumerous.

In this article, we give a wide generalization of the previous result using the results from \cite{FV}. Namely, we give a new interpretation of the set of reduced expressions of a permutation from which we define many families of tableaux having similar combinatorial properties as those of balanced tableaux. As a corollary of our method, we provide an alternative proof that balanced and standard tableaux are equinumerous. 

In order to do so, we deal with a bigger class of tableaux which are not required to be standard or balanced. To each tableau $T$ of shape $S$, where $S$ is any finite subset of $\mathbb{N}\times \mathbb{N}$, we associate a combinatorial object $\T$ called the \emph{type} of $T$. This allows us to split the set of tableaux of shape $S$ into different classes: two tableaux being in the same class if and only if they have same type. Let us denote by $\Tab(\T)$ the set of all the tableaux having type $\T$. In particular, both sets of standard and balanced tableaux are special instances of this classification. We also provide an algorithmic process allowing us to construct all tableaux having a given type, and in particular this allows us to easily obtain all balanced tableaux of a given shape.

In Section~\ref{SubSecTypPerm}, we associate each permutation $\sigma \in S_n$ with a type $\T_{\sigma}$ such that we have a one-to-one correspondence between $\Tab(\T_{\sigma})$ and $\Red(\sigma)$, where $\Red(\sigma)$ is the set of reduced expressions of $\sigma$. This gives a new combinatorial interpretation of reduced expressions of any permutation (not fundamentally different from the one given in \cite{FGR}, but maybe more natural with respect to the results in \cite{FV}). Then, we focus on the case of vexillary permutations, namely permutation being $2143$-avoiding, which are one of the main objects studied in \cite{FGR}. It is well-known \cite{S1} that each vexillary permutation $\sigma$ is associated with a partition $\lambda(\sigma)$ such that $|\Red(\sigma)|$ equals the number of standard tableau of shape $\lambda(\sigma)$. We then introduce a transformation on types called the \emph{exchange algorithm} (see Section~\ref{SectionExchange}), and we use it to prove the following theorem. \\

\noindent \textbf{Theorem \ref{TheoVexiType}.} \emph{Let $\sigma \in S_n$ be a vexillary permutation and $\T_{\sigma}^E$ be the type obtained by performing the exchange algorithm on $\T_{\sigma}$. Then, $|\Tab(\T_{\sigma}^E)|$ equals the number of balanced tableaux of shape $\lambda(\sigma)$, and each element of $\Tab(\T_{\sigma}^E)$ is of shape $\lambda(\sigma)$.} \\

This result provides a generalization of Edelman and Greene result. Indeed, we have the following two facts:
\begin{itemize}
\item from Theorem~\ref{TheoVexiType} can be deduced that balanced and standard tableaux are equinumerous (see Section~\ref{SecBalanced});
\item for any vexillary permutations $\sigma \in S_n$ and $\omega \in S_m$, we have $\T_{\sigma}^E=\T_{\omega}^E$ if and only if $\sigma$ can be obtained from $\omega$ by adding or deleting some fix points at its end and beginning (see Section~\ref{SecEquivRelat}).
\end{itemize}

We finish our study with exhibiting some combinatorial properties of the types $\T_{\sigma}^E$ (see Section~\ref{SecCombinProp}). We explain how a new combinatorial description of Schur functions arise from them, and we enumerate the elements of $\Tab(\T_{\sigma}^E)$ such that the integers $1,\ldots,k$ appear in given fixed positions. This last proposition provides a direct proof of \cite[Equation~(2.4)]{EG} as asked by Edelman and Greene.

\section{Definitions, notations and background}\label{SecDefBack}
In this section, we recall some basic definitions and background about standard and balanced tableaux.

A partition $\lambda$ of a nonnegative integer $n \in \mathbb{N}$ is a nonincreasing sequence of nonnegative integers $\lambda_1 \geq \lambda_2 \geq \cdots$ such that $\sum \lambda_i =n$. The integers $\lambda_i \neq 0$ are called parts of the partition $\lambda$. The Ferrers diagram of $\lambda$ is a finite collection of boxes, or cells, arranged in left-justified rows of lengths given by the parts of $\lambda$. By flipping this diagram over its main diagonal, we obtain the diagram of the conjugate partition of $\lambda$, denoted by $\lambda'$. We usually identify a partition with its Ferrers diagram.

More generally, in this article we work with diagrams of arbitrary shape, namely finite subsets of $\mathbb{N} \times \mathbb{N}$, without any constrain: let $S \subset \mathbb{N} \times \mathbb{N}$ such that $|S|=n$ (where $|S|$ denote the cardinal of $S$). We identify $S$ with a set of boxes in the plan, using the English convention for the coordinates of each box (\emph{i.e.} we use ``matrix-like coordinates''). A \emph{tableau} $T$ of shape $S$ is a bijective filling of $S$ (seen as a set of boxes) with entries in $[n]:=\{ 1,2,\ldots,n \}$. Given a tableau we denote its shape by $\Sh (T)$. If we require $\Sh (T)$ to be a partition $\lambda$, then $T$ will be what is usually called a \emph{Young tableau}. Moreover, if we consider Young tableaux satisfying the conditions that the filling is 
\begin{enumerate}
\item increasing from left to right across each row;
\item increasing down each column;
\end{enumerate} 
we obtain the set of \emph{standard Young tableaux} of shape $\lambda$, denoted by ${\rm SYT}(\lambda)$. 

\begin{definition}
Let $S$ be a diagram and $\mathfrak{c}=(a,b)$ be a box of $S$. We define the following sets,
 \begin{equation}\label{eq1}
L_{S}(a,b) = \{ (k,b) \ | \ k \geq a, (k,b) \in S \}, \ A_{S}(a,b)= \{ (a,k) \ | \ k>b, (a,k) \in S \}, 
\end{equation}
\begin{equation}\label{eq2}
H_{S}(a,b) = A_{S}(a,b) \biguplus L_{S}(a,b),
\end{equation}
respectively called the \emph{leg}, the \emph{arm}, and the \emph{hook} based on $(a,b)$.
We will denote by $l_{S}(a,b)$, $\mathtt{a}_{S}(a,b)$, and $h_{S}(a,b)$ their respective cardinalities.
\end{definition}

This notion of hook allows us to enumerate standard Young tableaux, thanks to the well-known \emph{hook-length formula} (see \cite{S2} for more details about this formula).

\begin{Theorem}
Let $\lambda$ be a partition of the integer $n$, seen as a diagram. Then, we have 
\[ |{\rm SYT}(\lambda)|=\frac{n!}{\prod_{(a,b) \in \lambda} h_{\lambda}(a,b)}. \]
\end{Theorem}

In \cite{EG}, Edelman and Greene introduced the concept of balanced tableaux, defined as follows.

\begin{definition}
Let $S$ be a diagram such that $|S|=n$. A \emph{balanced tableau} $T=(t_{a,b})_{(a,b)\in S}$ of shape $S$ is a Young tableau satisfying the following condition:
\[\text{for all} \ (a,b) \in S, \ \mathtt{a}_{a,b}=|\{(x,y) \in H_{S}(a,b) \ | \ t_{x,y}<t_{a,b}  \}|. \]
 We denote by $\Bal (S)$ the set of all balanced tableaux of shape $S$.
\end{definition}

In \cite{EG} the authors proved the following result about combinatorics of balanced tableaux.

\begin{Theorem}[\cite{EG}, Theorem~2.2]\label{TheoEGBal}
Let $\lambda$ be a partition of $n$. Then, we have 
\[
|\Bal(\lambda)|=|\text{SYT}(\lambda)|=f^{\lambda}.
\]
\end{Theorem}

The original proof is quite involved, and an alternative one is given in \cite{FGR}, which we now detail. In order to do so, we need to introduce the notion of reduced decomposition of a permutation and of vexillary permutation. It is classical that the symmetric group $S_n$ is generated by the simple transpositions $s_i=(i,i+1)$, $i \in [n-1]$, exchanging the positions of the integers $i$ and $i+1$. We denote by $\ell(\sigma)$ the minimal integer such that $\sigma$ can be written as a product of $\ell(\sigma)$ simple transpositions, and we define the reduced decompositions of any $\sigma \in S_n$ to be the elements of the set of word on $\{s_1,\ldots,s_{n-1}\}$
\[
\Red(\sigma):=\{s_{i_1}\cdots s_{\ell(\sigma)} \ | \ s_j \in S \ \text{and} \ s_{i_1}\cdots s_{\ell(\sigma)}=\sigma \}.
\]
Reduced decompositions are closely related to a partial order on $S_n$, called the (right) weak order and denoted by $\leq_R$. The weak order is defined as the transitive and reflexive closure of the covering relations 
\[
\text{for all} \ \sigma,\omega\in S_n, \ \sigma \lhd_R \omega \ \text{if and only if there exists} \ i\in [n-1] 
\]
\[
 \text{such that} \ \sigma=\omega.s_i \ \text{and} \ \ell(\omega)=\ell(\sigma)+1.
\]
In this context, there is a clear one-to-one correspondence between reduced decompositions of $\sigma$ and maximal chains from $Id$ to $\sigma$ in the poset $(S_n,\leq_R)$. There is also an alternative description of this poset in terms of \emph{inversion sets}. For all $\sigma \in S_n$, we define its inversion set to be
\[
\Inv(\sigma):=\{ (a,b) \ | \ 1 \leq a<b \leq n \ \text{and} \ \sigma^{-1}(a)> \sigma^{-1}(b) \}.
\]
It is classical that for any $\sigma,\omega \in S_n$, we have that $\sigma \leq_R \omega$ if and only if $\Inv(\sigma) \subseteq \Inv(\omega)$ (see, for instance, \cite{BB}).

We now define vexillary permutations.

\begin{definition}\label{DefVexillary}
Let $\sigma \in S_n$, we denote by $(d_i(\sigma))_i$ and $(g_i(\sigma))_i$ the finite sequences defined by
\begin{itemize}
\item $d_i(\sigma):=| \{ j>i \ | \ \sigma(j)<\sigma(i) \}|$,
\item $g_i(\sigma):=| \{ j<i \ | \ \sigma(j)>\sigma(i) \}|$.
\end{itemize}
We denote by $\mu(\sigma)$ and $\lambda(\sigma)$ the partitions obtained by rearranging in a nonincreasing order the sequences $(d_i)_i$ and $(g_i)_i$, respectively. We say that $\sigma$ is \emph{vexillary} if and only if $\lambda(\sigma)=\mu'(\sigma)$.
\end{definition}

In \cite{S1}, Stanley proved the following result using symmetric functions, giving an explicit formula to compute the number of reduced decompositions of any vexillary permutation.

\begin{Theorem}[Stanley, \cite{S1}]\label{TheoStanVexil}
Let $\sigma \in S_n$, if $\sigma$ is vexillary then
\[
|\Red(\sigma)|=f^{\lambda(\sigma)}.
\]
\end{Theorem}

We are now able to explain the proof of Theorem~\ref{TheoEGBal} that can be found in \cite{FGR}, which uses Theorem~\ref{TheoStanVexil} as fundamental tool. The first step consists in associating a diagram to each permutation.

\begin{definition}
Let $\sigma \in S_n$, the \emph{Rothe diagram} $\textbf{D}(\sigma)$ of $\sigma$ is the subset of $[n]\times [n]$ defined by
\[
\textbf{D}(\sigma):=\{(a,\sigma(b)) \in [n] \times [n] \ | \ a<b \ \text{and} \ \sigma(a)>\sigma(b) \}.
\]
\end{definition}

\begin{figure}[!h] 
\includegraphics[width=9cm]{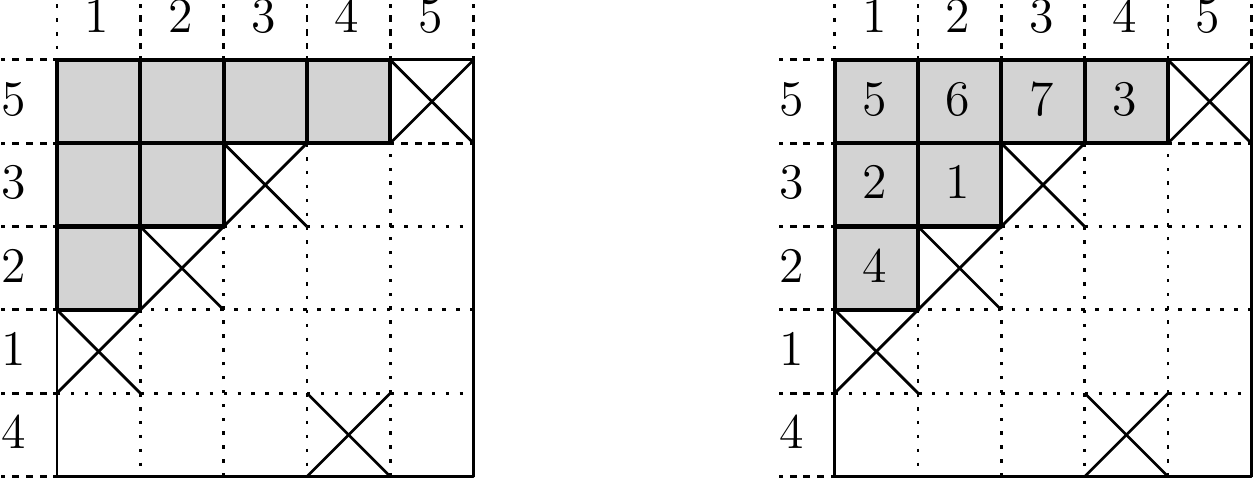}
\caption{Rothe diagram of the permutation $\sigma=[5,3,2,1,4]$ (on the left) and a balanced tableau of shape $\textbf{D}(\sigma)$ (on the right).}\label{FigRothEx}
\end{figure}

As it is shown in \cite{FGR}, the balanced tableaux whose shape is the Rothe diagram of a given permutation $\sigma$ (also called \emph{balanced labellings} of $\textbf{D}(\sigma)$) are intimately related to the reduced decompositions of $\sigma$.

\begin{Theorem}[\cite{FGR}, Theorem~2.4]
Let $\sigma \in S_n$ and $\emph{\textbf{D}}(\sigma)$ be its Rothe diagram. Then, there is a bijection between the set of the reduced decompositions of $\sigma$ and the set of the balanced tableaux of shape $\emph{\textbf{D}}(\sigma)$.
\end{Theorem}

Let us briefly explain how this bijection is constructed. By definition of the Rothe diagram of a permutation $\sigma \in S_n$, it is clear that 
\[
(a,\sigma(b)) \in \textbf{D}(\sigma) \ \text{if and only if} \ (a,b) \in \Inv(\sigma).
\]
Thus, a balanced tableau corresponds to an ordering $[(a_1,b_1),\ldots,(a_{\ell(\sigma)},b_{\ell(\sigma)})]$ of the inversions of $\sigma$. Furthermore, the authors proved in \cite{FGR} that $\{(a_1,b_1),\ldots,(a_i,b_i)\}$ is the inversion set of a permutation $\sigma_i \in S_n$ for all $1 \leq i \leq \ell(\sigma)$. Therefore, we have 
\[
Id \lhd_R \sigma_1 \lhd_R \sigma_2 \lhd_R \ldots \lhd_R \sigma_{\ell(\sigma)}=\sigma,
\]
\emph{i.e.}, a balanced tableau of shape $\textbf{D}(\sigma)$ corresponds to a maximal chain from $Id$ to $\sigma$ in the weak order on $S_n$. Thus, it corresponds to a reduced decomposition of $\sigma$, and it is proved in \cite{FGR} that this correspondence is bijective.

Eventually, this correspondence leads to a proof of Theorem~\ref{TheoEGBal}. 

\begin{Theorem}[\cite{FGR}, Theorem~3.4]\label{TheoFGRexistVexy}
Let $\lambda$ be a partition of $n$. Then, there exists a vexillary permutation $\sigma \in S_k$ for some $k \in \mathbb{N}$ such that 
\begin{itemize}
\item $\lambda(\sigma)=\lambda$;
\item the shape of $\emph{\textbf{D}}(\sigma)$ is $\lambda$ (up-to the deletion of some empty columns).
\end{itemize}
\end{Theorem}

Combining Theorem~\ref{TheoStanVexil} and Theorem~\ref{TheoFGRexistVexy}, Theorem~\ref{TheoEGBal} follows immediately. \\
In the sequel of this article, we will generalize Theorem~\ref{TheoEGBal} by exhibiting many families of tableaux having similar combinatorial properties as those of balanced tableaux. 

\section{Type of a tableau, definition and general properties}\label{SecDefTypes}

In this section, we introduce a natural generalization of balanced tableaux. We then study some of the properties of these generalizations and and mention some of the questions that naturally arise from this concept.

\subsection{Definition of types and filling algorithm}

In this section, $S$ will denote a diagram without any constraint on its shape.

\begin{definition}
A \emph{type} of shape $S$ is a filling of $S$ with integers $(\theta(\mathfrak{c}))_{\mathfrak{c} \in S}$ satisfying the following condition: for all $\mathfrak{c} \in S$, 
\[ 
0 \leq \theta(\mathfrak{c}) \leq h_{\mathfrak{c}}(S)-1.
\]
We denote by $\Typ(S)$ the set of all the type being of shape $S$.
\end{definition}

In the following definition, we explain how one can associate a type to each tableau of a given shape.

\begin{definition}\label{DefTypTab}
Let $T=(t_{\mathfrak{c}})_{\mathfrak{c}\in S}$ be a tableau of shape $S$. The type of $T$ is the type $\T=(\theta(\mathfrak{c}))_{\mathfrak{c}\in S}$ such that for all $\mathfrak{c} \in S$,
\[
\theta(\mathfrak{c})=| \{ \mathfrak{d} \in H_S(\mathfrak{c}) \ | \ t_{\mathfrak{d}} < t_{\mathfrak{c}} \}|.
\]
We denote by $\Tab_S(\T)$ the set of all tableaux of shape $S$ whose type is $\T$. When there is no ambiguity, we simply denote this set by $\Tab (\T)$.
\end{definition}

\begin{example}\label{ExDefBalStand}
Both balanced and standard tableaux are special instances of this classification: let $\B=(\mathtt{a}_{\mathfrak{c}}(S))_{\mathfrak{c} \in S}$ be the type whose each box is filled by its arm length, and $\mathcal{S}t \in \Typ(S)$ be the type whose each box is filled by the integer 0. By definition, we clearly have that $\Tab(\B)=\Bal(S)$. Moreover, if $T \in \Tab(\mathcal{S}t)$, then each integer appearing in $T$ is minimal in its associated hook. Thus, the entries in $T$ are increasing from left to right along each row and decreasing from top to bottom along each column, so that $\Tab(\mathcal{S}t)$ is the set of standard tableaux of shape $S$.
\end{example}

Definition~\ref{DefTypTab} provides a way to classify all tableaux according to their type. However, at this stage it is unclear if $\Tab(\T)$ is empty or not for any given type $\T$. In what follows, we give a combinatorial way to construct all the elements of $\Tab(\T)$, which will implies that $\Tab(\T)\neq \emptyset$ for any $\T \in \Typ(S)$. We begin with a useful definition.

\begin{definition}
Let $\T=(\theta(\mathfrak{c}))_{\mathfrak{c} \in S} \in \Typ(S)$ and $\mathfrak{c} \in S$. We say that $\mathfrak{c}$ is erasable in $\T$ if and only if the following two conditions are satisfied:
\begin{enumerate}
\item $\theta(\mathfrak{c})= 0$;
\item for all $\mathfrak{d} \in S \setminus \{\mathfrak{c} \}$ such that $\mathfrak{c} \in H_S(\mathfrak{d})$, we have $\theta(\mathfrak{d}) \neq 0$.
\end{enumerate}
\end{definition}

Clearly, if we consider an erasable box $\mathfrak{c}$ of a type $\T$ of shape $S$, if we suppress the box $\mathfrak{c}$ and we decrease by one the integer in each box $\mathfrak{d} \in S \setminus \{\mathfrak{c}\}$ such that $\mathfrak{c} \in H_S(\mathfrak{d})$, then what we obtain is a type of shape $S \setminus \{\mathfrak{c}\}$. This fact allows us to construct all the elements of $\Tab(\T)$, as it is explained after the following technical lemma.

\begin{lemma}\label{LemEras}
Let $\T=(\theta(\mathfrak{c})) \in \Typ(S)$. Then, there exists a box $\mathfrak{c}$ which is erasable in $\T$.
\end{lemma}

\begin{proof}
Since $S$ is finite, there exists a box $\mathfrak{d}$ such that $H_S(\mathfrak{d})=\{ \mathfrak{d}\}$. Thus, we have $\theta(\mathfrak{d})=0$. Let us now construct a sequence $(\mathfrak{d}_i)_{i\geq 1}$ of pairwise distinct boxes of $S$ such that $\mathfrak{d}_1=\mathfrak{d}$, $\theta(\mathfrak{d}_j)=0$ for all $j \geq 1$, and $\mathfrak{d}_j \in H_S(\mathfrak{d}_{j+1})$, and assume that the length $k$ of this sequence is maximal. By maximality, for all $\mathfrak{c}$ such that $\mathfrak{d}_k \in H_S(\mathfrak{c})$ we have $\theta(\mathfrak{c}) \neq 0$. Consequently, $\mathfrak{d}_k$ is erasable in $\T$, and this concludes the proof. 
\end{proof}

\begin{definition}[Filling process]
Let $\T \in \Typ(S)$ and $L=[\mathfrak{c}_1,\ldots,\mathfrak{c}_{|S|}]$ be a sequence of pairwise distinct boxes of $S$. We say that $L$ is \emph{filling sequence} of $\T$ if and only if there exists a sequence of types $(\T_i)_{1 \leq i \leq |S|}$ such that 
\begin{itemize}
\item $\T_1=\T$;
\item for all $1 \leq i \leq |S|$, $\mathfrak{c}_i$ is erasable in $\T_i$;
\item for all $1 \leq i < |S|$, $\T_{i+1}$ is obtained from $\T_i$ by suppressing $\mathfrak{c}_i$ in $\T_i$ and decreasing by one the valuation in each box $\mathfrak{d} \in S':=S\setminus \{c_1,\ldots,c_{i-1} \}$ such that $\mathfrak{c}_i \in H_{S'}(\mathfrak{d})$.
\end{itemize}
We associate each filling sequence $L=[\mathfrak{c}_1,\ldots,\mathfrak{c}_{|S|}]$ with a tableau $T_L=(t_{\mathfrak{c}})_{\mathfrak{c}}$ of shape $S$ defined by $t_{c_i}=i$ for all $1 \leq i \leq |S|$. 
\end{definition}

\begin{Remark}
Note that the filling process is a special instance of the \emph{peeling process} introduced in \cite{FV}.
\end{Remark}

\begin{proposition}
For any filling sequence $L$, we have $T_L \in \Tab(\T)$, and the map $L \mapsto T_L$ is a bijection. 
\end{proposition}

\begin{proof}
This is clear by induction on $|S|$.
\end{proof}

\begin{example} 
Consider the type $\T$ on the top-left of Figure~\ref{ExFilSeq} and the filling sequence $L=[(1,3);(1,1);(1,2);(2,2);(2,1)]$ (since we represent Ferrers diagrams with the English convention, we use the matrix coordinates for each box). The types on the top of the figure are the types obtained after each iteration of the filling process.
\begin{figure}[!h] 
\includegraphics[width=11.5cm]{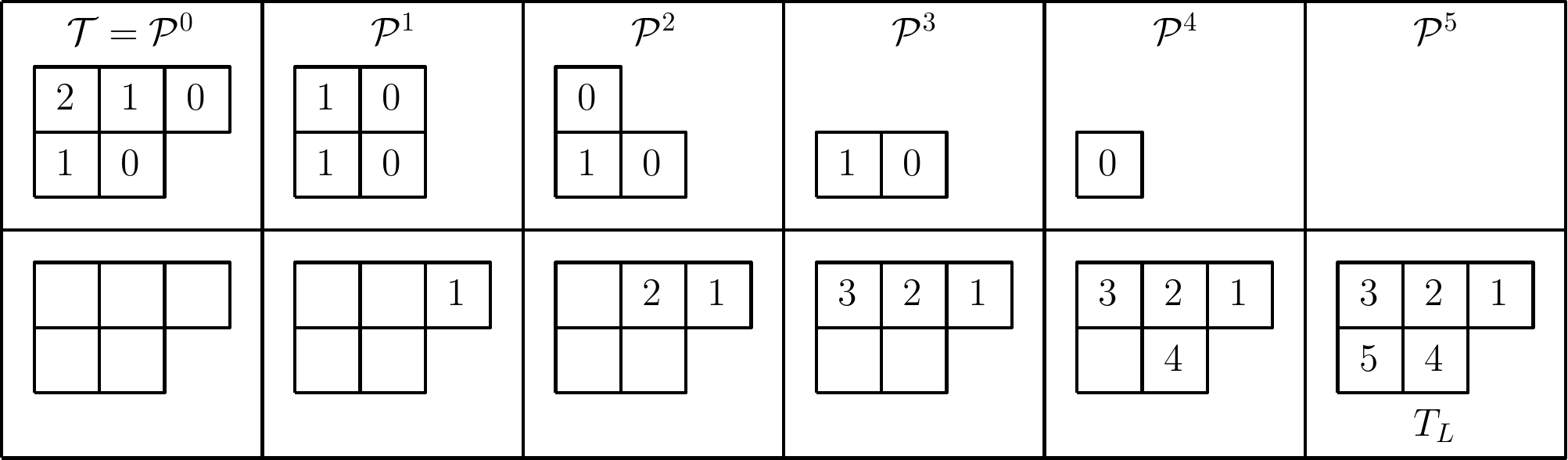}
\caption{}\label{ExFilSeq}
\end{figure}
\end{example}

Thanks to the previous proposition and Lemma~\ref{LemEras}, we also have the following result which concludes this section.

\begin{proposition}
For all $\T \in \Typ(S)$, we have $\Tab(\T) \neq \emptyset$.
\end{proposition}

\subsection{Some results about the enumeration of $\Tab(\T)$}

Now that we have a classification of all the tableaux of a given shape according to their type and a way to construct all tableaux in a given class, the following natural questions arise. 

\begin{Question}\label{QuestEnum}
$~~$

\begin{enumerate}
\item Is it possible to find a formula to compute $|\Tab(\T)|$ For any type $\T$?
\item At least, can we exhibit some family of types for which the number of corresponding tableaux can be computed?
\end{enumerate}
\end{Question}

Even if the general case seems to be quite difficult, we have some basic properties in that direction that we now detail.
First, note that if $\lambda=(n)$ or $\lambda=(1^n)$, then for any $\T \in \Typ (\lambda)$ there exists a unique tableau of type $\T$. This is clear, since at each iteration of the filling process there is only one $(i,j)\in \lambda$ which is erasable. This basic fact leads us to our first enumerative proposition, generalizing Lemma~3.2 from \cite{EG}. 

\begin{proposition}
Let $k$ and $p$ be two integers and $\T$ be a type of shape $\lambda=(k,1^p)$. Then, we have
\[
|\Tab(\T)|=f^{\lambda}.
\]
\end{proposition}

\begin{proof}
First, note that for any tableau $T=(t_{\mathfrak{c}})_{\mathfrak{c}\in \lambda}$ of type $\T$, we have $t_{1,1}=\theta(1,1)+1$ by definition. Thus, if we set $S:=\lambda\setminus \{ (1,1) \}$ (see Figure \ref{HookCase}) and $\T ' :=(G_S,\theta)$, then we have 
\[ |\Tab _{\lambda}(\T)|=|\Tab_{S}(\T')|. \]
\begin{figure}[!h]
\includegraphics[width=3.5cm]{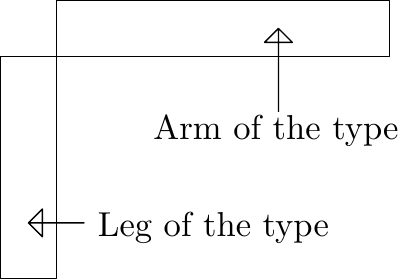}
\caption{}\label{HookCase}
\end{figure}
Moreover, when we perform the filling process on $\T'$, the only thing we have to chose at each step is an element in the leg or in the arm of $S$, and this is independent of the choice of $\T$. This concludes the proof.
\end{proof}

In general, finding an explicit formula for the number of tableaux of a given type $\T$ seems to be a quite complicated problem (note that an approach similar to that of \cite{HookProb,NPS} has not been tried yet). However, a probabilistic approach might be possible, as suggested by the following result.

\begin{proposition}
Let $S$ be a diagram, if we choose uniformly a type $\T$ in $\Typ(S)$, then the expected value for $| \Tab_{S}(\T)|$ is 
\[
\frac{n!}{\prod_{\mathfrak{c}\in S}h_{S}(\mathfrak{c})}.
\]
\end{proposition}

\begin{proof}
Set $H_{S}= \prod_{\mathfrak{c}\in \lambda}h_{S}(\mathfrak{c})$, it is clear that the number of types of shape $S$ is precisely $H_S$. Then, because of the uniform choice, the probability for a type $\T$ to be chosen is exactly $\frac{1}{H_S}$. Thus, the expected value for $|\Tab(\T)|$ is 
\[ \frac{\sum_{\T \in \Typ(S)} |\Tab_{S}(\T)|}{H_{S}}, \]
and the numerator clearly equals $n!$. The result follows.
\end{proof}

This last proposition leads us to the following natural question.

\begin{Question}
Is it possible to find an explicit formula for the variance ?
\end{Question}

This last question is open, and it seems that the value of the variance heavily depends on the shape of the considered diagram: for instance, some tests suggest that the variance is maximal when we consider a square shape.

\section{Types and reduced decompositions of permutations}

In this section, our objective is twofold: we give a positive answer to Question~\ref{QuestEnum}~(2), and we provide a wide generalization of Theorem~\ref{TheoEGBal}. More precisely, we explain how to associate each vexillary permutation $\sigma$ with a type $\T_{\sigma}^E$ of shape $\lambda_{\sigma}$ (we will explain in Section~\ref{SectionExchange} what the index ``$E$'' stands for) such that: 
\begin{itemize}
\item $|\Tab(\T_{\sigma}^E)|=f^{\lambda(\sigma)}$;
\item the map $\sigma \mapsto \T_{\sigma}^E$ is injective up-to an explicit and simple equivalence relation (see Section~\ref{SecEquivRelat});
\item elements of $\Tab(\T_{\sigma}^E)$ share many combinatorial properties similar to those of balanced tableaux (see Section~\ref{SecCombinProp}).
\end{itemize} 

\subsection{Type associated with a permutation}\label{SubSecTypPerm}

We first explain how one can associate any permutation $\sigma$ with a type $\T_{\sigma}$ such that $|\Tab(\T_{\sigma})|=|\Red(\sigma)|$, using the results from \cite{FV}.

Let $\lambda_n$ denote the \emph{staircase partition} $(n-1,n-2,\ldots,1)$. We identify the Ferrers diagram of $\lambda_n$ with the set $\{(a,b) \in \mathbb{N}| 1 \leq a<b \leq n\}$, by choosing new coordinates for each box of $\lambda_n$, as depicted on Figure~\ref{FigTypSym}. Thanks to these coordinates, we associate to each box $(a,b) \in \lambda_n$ the integer $\theta(a,b)=b-a-1$, and this defines a type $\mathcal{A}_n$ of shape $\lambda_n$, as depicted on the right of Figure~\ref{FigTypSym}.

\begin{figure}[!h] 
\includegraphics[width=9cm]{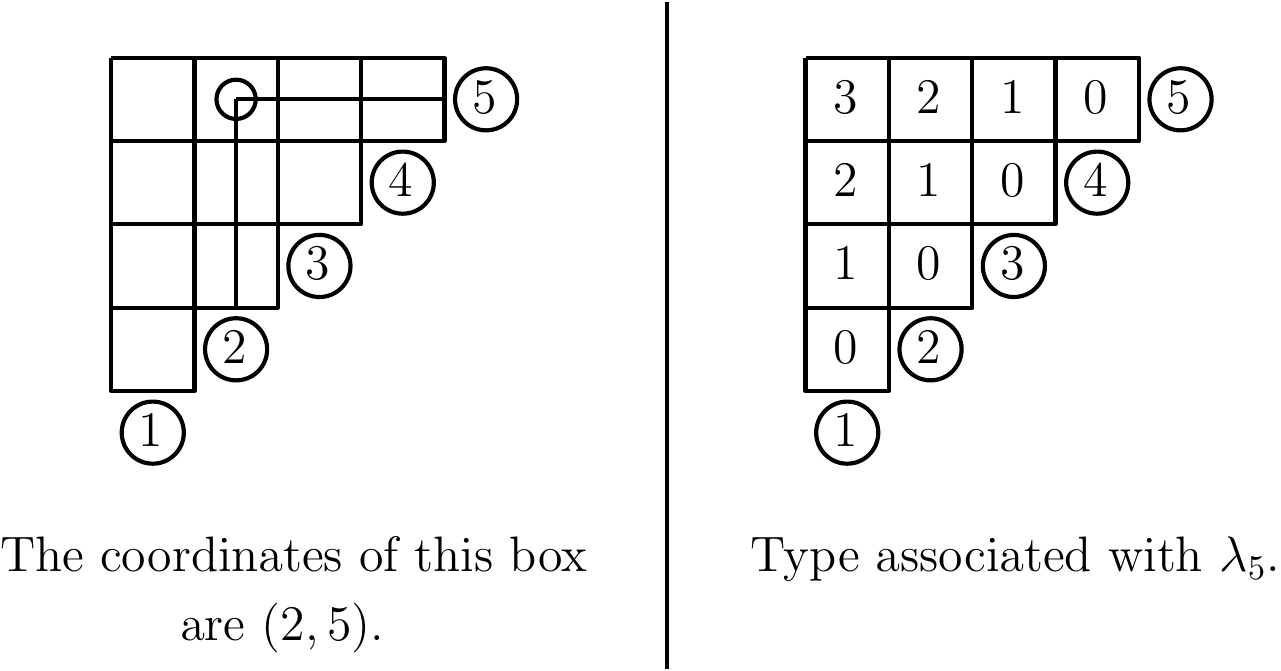}
\caption{Representation of the type $\mathcal{A}_n$.}\label{FigTypSym}
\end{figure}

Using the type $\A_n$, we can associate each permutation with a type, thanks to the following definition and proposition.

\begin{definition}
We denote by $\T_{\sigma}$ the sub-diagram of $\lambda_n$ made of the boxes whose coordinates are the elements of $\Inv(\sigma)$, and such that each box $\mathfrak{c} \in \Inv(\sigma)$ is filled with the integer $\theta(\mathfrak{c})$ coming from the definition of $\A_n$.
\end{definition}

\begin{proposition}
Let $\sigma \in S_n$. Then, $\T_{\sigma}$ is a type and we have 
$$|\Tab(\T_{\sigma})|=|\Red(\sigma)|. $$
\end{proposition}

\begin{proof}
This is an immediate reformulation of \cite[Theorem~4.1]{FV}.
\end{proof}

\begin{Remark}
Notice that it is possible to construct a one-to-one correspondence between elements of $\Tab(\T_{\sigma})$ and balanced tableaux of shape $D(\sigma)$, by swapping positions of some rows in elements of $\Tab(\T_{\sigma})$. Therefore, it is possible to reformulate what follows in terms of Rothe diagram and their labbelings. However,  using Rothe diagrams instead of types $\T_{\sigma}$ would require to rewrite most of the results from \cite{FV}, without major modification. Furthermore, the use of the types $\T_{\sigma}$ seems to be more natural, since they have similar properties as those of Rothe diagrams, and provide a natural description of the weak order and its combinatorics.
\end{Remark}

\subsection{A transformation on types}

At this point, we already have a way to associate each vexillary permutation $\sigma$ with a type $\T_{\sigma}$ such that $|\Tab(\T_{\sigma})|=f^{\lambda(\sigma)}$. However, in general the shape of $\T_{\sigma}$ is not the Ferrers diagram of $\lambda(\sigma)$. In the sequel, we introduce a combinatorial transformation on types that will allow us to turn $\T_{\sigma}$ into a type of shape $\lambda(\sigma)$ whenever $\sigma$ is vexillary.

Let us begin this section with introducing two notations.

\begin{definition}
Let $S$ be a diagram and $``a"$ (resp. $``b"$) be a row (resp. a column) of $S$. We denote by $S\! \! \downarrow _a$ (resp. $\overrightarrow{S}^b$) the diagram obtained by swapping rows $a$ and $a+1$ (resp. columns $b$ and $b+1$) of $S$.
\end{definition}

\begin{definition}
 Let $T$ be a tableau of shape $S$ and $a$ (resp. $b$) a row (resp. a column) of $S$. We denote by $T \! \! \downarrow _a$ (resp. $\overrightarrow{T}^b$) the tableau of shape $S\! \! \downarrow _a$ (resp. $\overrightarrow{S}^b$) obtained from $T$ by exchanging rows $a$ and $a+1$ (resp. columns $b$ and $b+1$).
\end{definition}

Let us consider a type $\T$ of shape $S$ and let $a$ be the index of a row of $S$. In general, the set 
$A=\{T \! \! \downarrow _a \ | \ T \in \Tab(\T) \}$
does not correspond to a class of our classification. That is, in general there is no type $\T'$ of shape $S\! \! \downarrow _a$ such that $A=\Tab(\T')$. However, we will prove in the sequel of this section that such a type $\T'$ exists in a specific case.

\begin{definition}
Let $S$ be a diagram, $\T\in \Typ(S)$ and $``a"$ be the index of a row of $\T$. We say that the row $a$ is \emph{dominant} if and only if 
\begin{itemize}
\item for all $(a,y) \in \mathbb{N} \times \mathbb{N}$, if $(a,y) \in S$, then $(a+1,y) \in S$;
\item  for all $(a,y) \in S$ we have $\theta(a,y)>\theta(a+1,y)$.
\end{itemize}
We have a similar definition of \emph{dominant column} (see Figure~\ref{FigExampleDomin}) for a graphical representation of these two notions).
\end{definition}

\begin{figure}[!h]
\includegraphics[width=8cm]{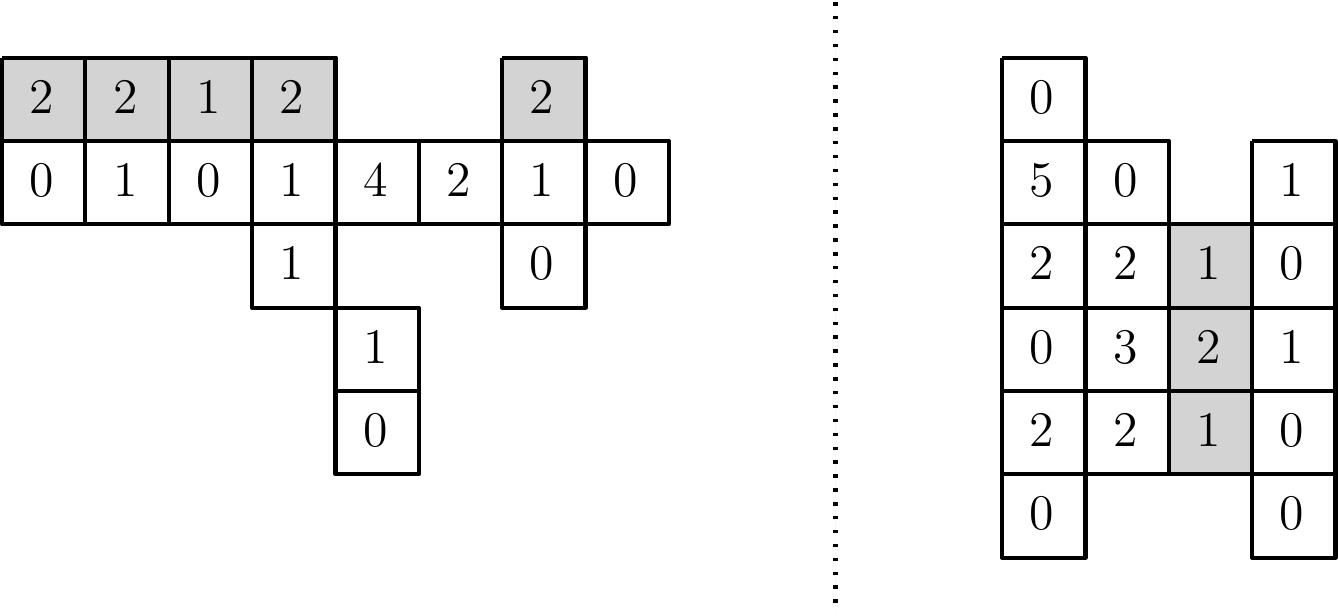}
\caption{A dominant row (on the left) and a dominant column (on the right).}\label{FigExampleDomin}
\end{figure}

Before moving to the combinatorial study of the types having a dominant row or column, let us introduce one last notation.

\begin{definition}
Let $\T$ be a type of shape $S$ and $``a"$ be the index of a dominant row of $\T$. We denote by  $\T \! \! \downarrow _a$ the type of shape $S \! \! \downarrow _a$ being obtained from $\T$ by first decreasing by one all the integers in the row $a$ of $\T$, then by swapping rows $a$ and $a+1$ (resp. columns $b$ and $b+1$) of $\T$, and keeping all other entries unchanged (see Figure \ref{ExchExpClear}).
\end{definition}

\begin{figure}[!h]
\includegraphics[width=10cm]{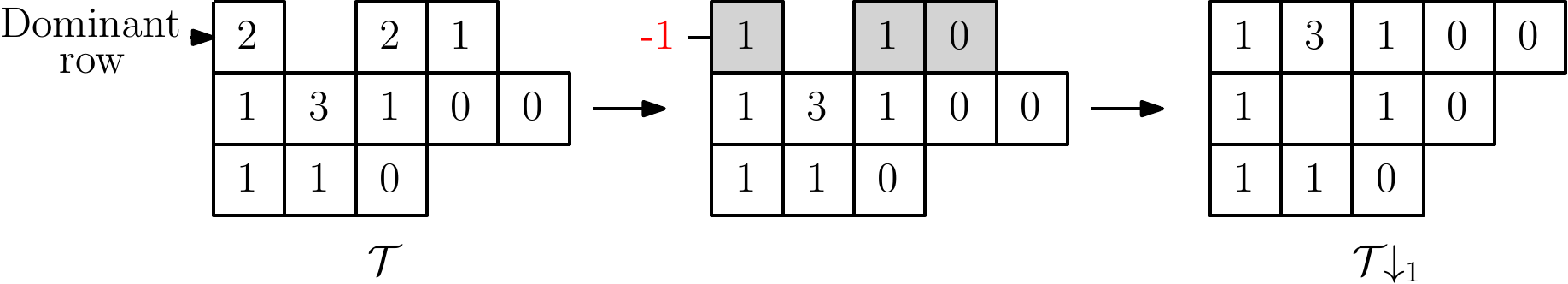}
\caption{From $\T$ to $\T \!\! \downarrow _1$.}\label{ExchExpClear}
\end{figure}

Our aim is now to prove that for any dominant row $a$ of a type $\T$ we have
\begin{equation}\label{EqPropDomExchange}
\{ T \! \! \downarrow _a \ | \ T \in \Tab(\T) \}=\Tab(\T \! \! \downarrow _a).
\end{equation}
For that purpose, we first prove a technical lemma.

\begin{lemma}\label{StepEx}
Let $\T=(G_S,\theta) \in \Typ(S)$ and $a$ be the index of a dominant row of $\T$. Then, for any $T=(t_{c}) \in \Tab(\T)$, we have $t_{a,y}>t_{a+1,y}$ for all $(a,y) \in S$.
\end{lemma}

\begin{proof}
Let $T=(t_{(x,y)})_{(x,y) \in S} \in \Tab(\T)$, and assume by contradiction that the lemma is not true and consider $y$ maximal such that $t_{(a,y)}\leq t_{(a+1,y)}$. Let $\mathfrak{c} \in H_{S}(a,y) \setminus \{ (a,y) \}$ such that $t_{\mathfrak{c}}<t_{(a,y)}$, and let us split our study into two cases. 
\begin{itemize}
\item If $\mathfrak{c} \in L_S(a,y)$, then we have that $\mathfrak{c} \in H_S(a+1,y)$ and $t_{\mathfrak{c}}<t_{(a,y)} \leq t_{(a+1,y)}$.
\item If $\mathfrak{c} \in A_S(a,y)$, then there exists $z > y$ such that $\mathfrak{c}=(a,z)$, and we have by maximality of $y$
\[
t_{(a+1,y)} \geq t_{(a,y)} > t_{(a,z)} > t_{(a+1,z)}.
\]
\end{itemize}  
This is enough to show that $\theta(a,y) \leq \theta(a+1,y)$, and this contradicts the fact that $a$ is dominant. This concludes the proof.
\end{proof}

We now prove that \eqref{EqPropDomExchange} holds.

\begin{proposition}[Exchange property]\label{LiEx}
Let $\T$ be a type and $``a"$ (resp. $b$) be a dominant row (resp. column) of $\T$. Then, the map $T \mapsto T \! \! \downarrow_a$ (resp. $T \mapsto \overrightarrow{T}^b$) is a bijection between $Tab(\T)$ and $Tab(\T\! \! \downarrow _a)$ (resp. $\Tab(\overrightarrow{\T}^b)$).
\end{proposition}  

\begin{proof}
Let $T \in \Tab(\T)$, and denote by $\T'=(\theta'(\mathfrak{c}))_{\mathfrak{c} \in S \! \downarrow_a}$ the type of the tableau $T':= T\! \!  \downarrow_a=(t_{x,y}')$. We will prove that $\T'=\T \! \! \downarrow_a$.

Let $(x,y)$ be a box of $S \! \! \downarrow_a$ and let us define the following set
$$H_{x,y}(T):=\{ t_{a,b} \ | \ (a,b) \in H_{x,y}(\Sh (T) ) \ \}. $$ 
We split our study into three cases.
\begin{itemize}
\item If $x \notin  \{a,a+1\}$, then we have $H_{x,y}(T)=H_{x,y}(T')$, so that $\theta'(x,y)=\theta(x,y)$.
\item If $x=a$, then we have $H_{a,y}(T')=H_{a+1,y}(T) \cup  \{ t_{a,y} \}$. However, by Lemma \ref{StepEx} we have $t'_{a,y}=t_{a+1,y} < t_{a,y}$, so that $\theta'(a,y)=\theta(a+1,y)$.
\item If $x=a+1$, then we have $H_{a+1,y}(T')=H_{a,y}(T) \setminus  \{ t_{a+1,y} \}$, so that $\theta'_{a+1,y}=\theta_{a,y}-1$.
\end{itemize}
Then, we have $\T'=\T\! \! \downarrow _a$, hence $T \mapsto T \! \! \downarrow_a$ send an element of $\Tab_{S}(\T)$ to an element of $\Tab_{S \! \downarrow_a}(\T\! \! \downarrow _a)$. Similar arguments show that $T \mapsto T \! \! \downarrow_a$ also sends an element of $\Tab_{S \! \downarrow_a}(\T\! \! \downarrow _a)$ to an element of $\Tab_{S}(\T)$, but $\downarrow_a$ is an involution, so that it is 	 bijection. This concludes the proof for rows. The proof of the same property for columns is similar.
\end{proof}

We finish this section with a useful definition.

\begin{definition}
Let $\T$ be a type of shape $S$ and $a$ be the index of a row of $\T$. The row $a$ is called \emph{dethroned} if and only if
\begin{itemize}
\item for all $(a,y) \in \mathbb{N} \times \mathbb{N}$, if $(a,y) \in S$, then $(a-1,y) \in S$;
\item for all $(a,y) \in S$ we have $\theta(a-1,y)\leq \theta(a,y)$.
\end{itemize} 
We have a similar notion of \emph{dethroned column}.
\end{definition}

Obviously, if $a$ is a dominant row of $\T$, then $a+1$ is a dethroned line of $\T\! \! \downarrow _a$ and conversely. The same holds for dominant columns. If $a+1$ is a dethroned line of $\T$, we denote by $\T \! \! \uparrow_{a+1}$ the unique type such that $(\T \! \! \uparrow_{a+1}) \! \! \downarrow_a=\T$.

\subsection{The exchange algorithm}\label{SectionExchange}

In this section, we explain how one can turn the type $\T_{\sigma}$ (where $\sigma \in S_n$ is vexillary) into a type of shape $\lambda(\sigma)$ using recursively Proposition~\ref{LiEx} on lines and columns.

\begin{definition}[Line-exchange algorithm]
Let $\T$ be a type of shape $S$, the \emph{line-exchange algorithm} is the algorithm described below.
\begin{enumerate}
\item Erase all the empty rows of $\T$.
\item Set $i:=1$.
\begin{enumerate}
\item If $i$ is a dominant row of $\T$, then set $\T:=\T\! \! \downarrow_i$ and go back to step (2). Otherwise, go to step (2-b).
\item If there is no row below $i$, then the algorithm stops. Otherwise, set $i:=i+1$ and go back to step (2-a).
\end{enumerate}
\end{enumerate}
We denote by $\T^L$ the type obtained after we perform the line-exchange algorithm.
\end{definition}

There is an obvious analogous \emph{column-exchange algorithm}, and we denote by $\T^C$ the type obtained after we perform this algorithm on a type $\T$.

\begin{lemma}
For any type $\T$, we have 
\[ 
|\Tab(\T)|=|\Tab(\T^L)|=|\Tab(\T^C)|.
\]
\end{lemma}

\begin{proof}
It is clear by Proposition~\ref{LiEx}.
\end{proof}

\begin{definition}
For all type $\T$, we denote by $\T^E$ the type $(\T^L)^C$ obtained by first performing the line exchange algorithm on $\T$, and then performing the column exchange algorithm on $\T^L$.
\end{definition}

We now state the main result of this section, whose proof is detailed in Section~\ref{SectionProof}

\begin{Theorem}\label{TheoVexiType}
Let $\sigma \in S_n$ be a vexillary permutation and $\T_{\sigma}$ be its associated type. Then, we have:
\begin{enumerate}
\item $|\Tab(\T_{\sigma}^E)|=f^{\lambda(\sigma)}=f^{\lambda(\sigma)'}$;
\item The shape of $\T_{\sigma}^E$ is $\lambda(\sigma)'$.
\end{enumerate}
\end{Theorem}

\begin{figure}[!h] 
\includegraphics[width=3.5cm]{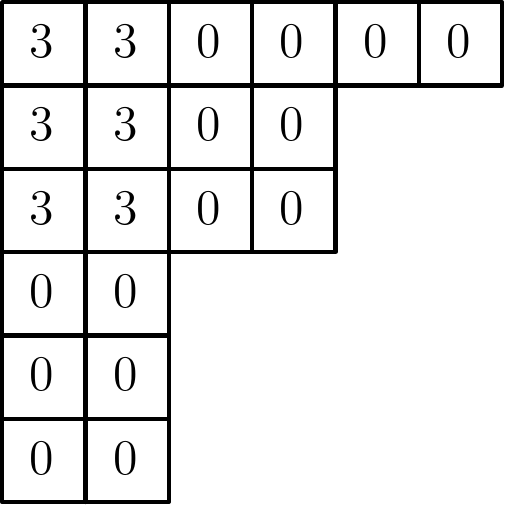}
\caption{This is the type $\T_{\sigma}^E$ obtained considering the vexillary permutation $\sigma=[4,8,9,5,7,6,1,3,2]$}
\end{figure}

\subsection{Proof of Theorem~\ref{TheoVexiType}}\label{SectionProof}

The first step of the proof consists in a characterization of vexillary permutations using their associated type.

\begin{definition}
Let $\sigma \in S_n$, we denote by $(l_i(\T_{\sigma}))_i$ and $(c_i(\T_{\sigma}))_i$ the sequences defined by
\begin{align*}
l_i(\T_{\sigma})&:=| \{ j \ | \ (j,i) \in \Inv(\sigma)  \}|, \\
c_i(\T_{\sigma})&:=|\{ j \ | \ (i,j) \in \Inv(\sigma)  \}|.
\end{align*}
\end{definition}

The following lemma is immediate by Definition~\ref{DefVexillary}.

\begin{lemma}\label{TempStack}
Let $\sigma \in S_n$, then the partition obtained by rearranging the sequence $(l_i(\T_{\sigma}))_i$ (resp. $(c_i(\T_{\sigma}))_i$) in a non-increasing order is $\mu(\sigma)$ (resp. $\lambda(\sigma)$).
\end{lemma}

Let us now consider $\sigma \in S_n$, we begin with putting the diagram $\Inv(\sigma)$ in a grid as depicted on Figure~\ref{SubDiagram}. 
\begin{figure}[!h] 
\includegraphics[width=5cm]{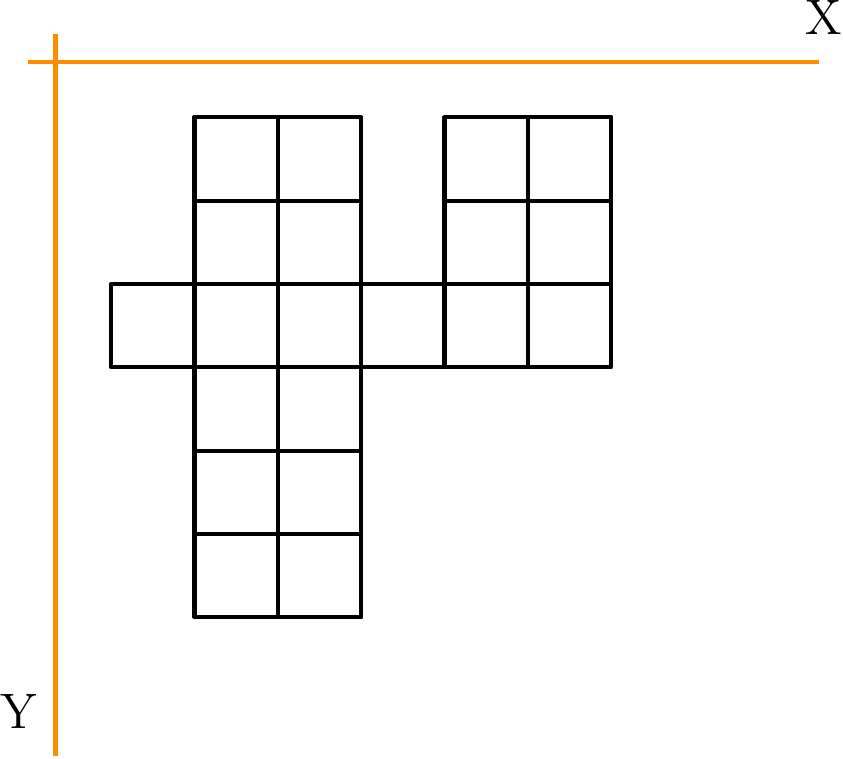}
\caption{Diagram associated with $\sigma=[7,8,4,5,1,2,6,9,3]\in S_9$}\label{SubDiagram}
\end{figure}
We first push all the boxes of $\Inv(\sigma)$ against the $Y$-axes, and we then push all the boxes against the $X$-axes, obtaining by this way a Ferrers diagram (see Figure~\ref{StackProcess}). 
\begin{figure}[!h] 
\includegraphics[width=9 cm]{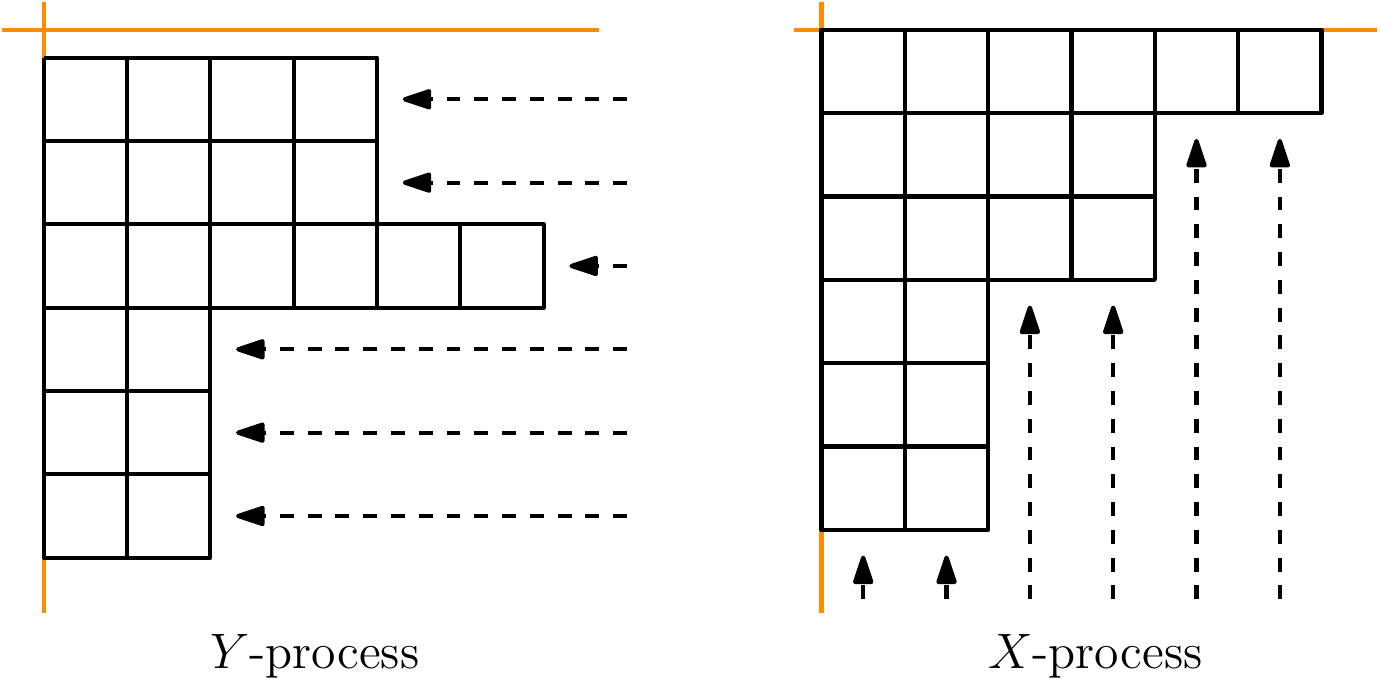}
\caption{}\label{StackProcess}
\end{figure}
The partition obtained after this \emph{$YX$-process} is $\mu(\sigma)$. 
Indeed, when we packed all the boxes against the $Y$-axes we obtain a diagram whose rows are left-justified, and row $i$ contains exactly $l_i(\T_{\sigma})$ boxes. Therefore, when we push everything on the $X$-axes, we are just rearranging these rows in a non-increasing order. Thus, thanks to Lemma~\ref{TempStack} the resulting diagram is precisely $\mu(\sigma)$.
Clearly, if we first stack on the $X$ and then on the $Y$-axes (this process is called the \emph{$XY$-process}), then the resulting partition is precisely $\lambda(\sigma)'$. 

The following proposition is an immediate consequence of the observation made in the previous paragraph.

\begin{proposition}\label{PropStack}
Let $\sigma \in S_n$. Then, $\sigma$ is vexillary if and only if the partitions obtained after we perform the $XY$-process and $YX$-process on $\Inv(\sigma)$ are the same.
\end{proposition}

We now prove an intermediate lemma.

\begin{lemma}\label{VexCarComb}
Let $\sigma \in S_n$ be a vexillary permutation and $i,j$ be two integers. Then, we have the following two properties.
\begin{itemize}
\item If $l_i(\T_{\sigma}) \leq l_j(\T_{\sigma})$, then we have that for all $(i,a) \in \mathbb{N} \times \mathbb{N}$, 
\[
\text{if} \ (a,i) \in \Sh(\T_{\sigma}), \ \text{then} \  (a,j) \in \Sh(\T_{\sigma}).
\]
\item If $c_i(\T_{\sigma}) \leq c_j(\T_{\sigma})$, then we have that for all $(a,i) \in \mathbb{N} \times \mathbb{N}$,
\[
\text{if} \ (i,a) \in \Sh(\T_{\sigma}), \ \text{then} \  (j,a) \in \Sh(\T_{\sigma}).
\]
\end{itemize}
\end{lemma}

\begin{proof}
We prove this lemma only for lines, since the proof for columns is similar, and we simply denote by $l_i$ the integer $l_i(\T_{\sigma})$.
We denote by $n$ the number of non-empty rows in the diagram $\Sh(\T_{\sigma})$ and we set $i_1,\ldots,i_n$ a sequence of indices such that:
\begin{itemize}
\item $l_{i_k} \neq 0$ for all $k \in [n]$;
\item the sequence  $(l_{i_1},\ldots, l_{i_n})$ is non-increasing.
\end{itemize} 
We will prove by induction on $k \in [n]$ that the property holds for the row $i_k$. First, notice that we have $l_{i_1} \geq l_{i_q}$ for all $1< q \leq n$. Let us fix such a $q$, and consider a box $\mathfrak{c}=(i_q,p) \in \Sh(\T_{\sigma})$. 

Assume by contradiction that $(i_1,p) \in \Sh(\T_{\sigma})$, then we have the configuration depicted on Figure~\ref{Fig100}.\begin{figure}[!h] 
\includegraphics[width=8 cm]{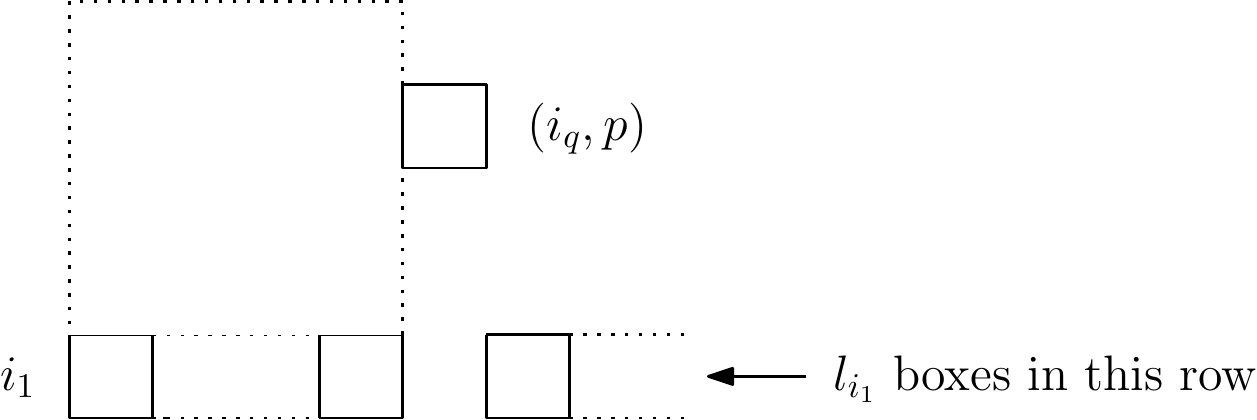}
\caption{Configuration when $(i_q,p) \in \Sh(\T_{\sigma})$ and $(i_{1},p)\notin \Sh(\T_{\sigma})$}\label{Fig100}
\end{figure}
Therefore, if we push the boxes against the $X$-axes, then in the first row there must be strictly more than $l_{i_1}$ boxes as represented on Figure~\ref{Fig101}. 
\begin{figure}[!h] 
\includegraphics[width=10 cm]{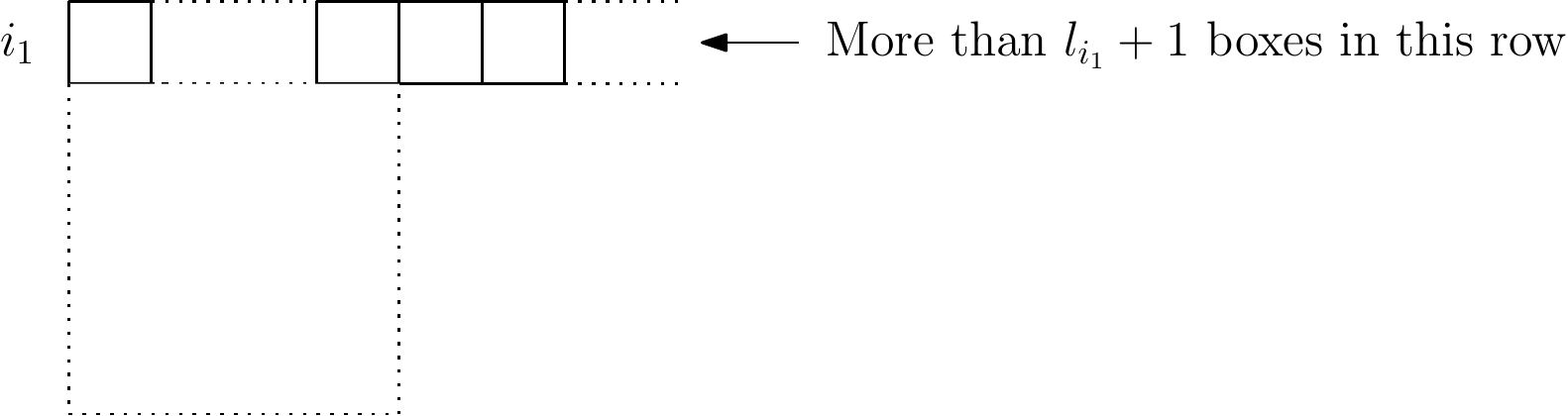}
\caption{Configuration after we pushed the boxes of $\Sh(\T_{\sigma})$ against the $X$-axes.}\label{Fig101}
\end{figure}
Thus, there are strictly more than $l_{i_1}$ boxes in the first row of the partition obtained after we perform the $XY$-process on $\Sh(\T_{\sigma})$. However, by maximality of $l_{i_1}$, the first row of the partition obtained we perform the $YX$-process on $\Sh(\T_{\sigma})$ contains $l_{i_1}$ boxes. Thus, $\lambda(\sigma) \neq \mu(\sigma)'$, and this contradicts the fact that $\sigma$ is vexillary. Consequently, we have $(i_1,p)\in \Sh(\T_{\sigma})$ and the lemma is true for row $i_1$. 

Let $k$ be such that the lemma is true for rows $i_1,\ldots,i_k$, and let $\lambda(\sigma)=(\lambda_1,\ldots,\lambda_m)$. By induction, if we delete rows $i_1,\ldots,i_k$ in $\Sh(\T_{\sigma})$ and then perform the $XY$ or $YX$ stacking process on the obtained diagram, then the resulting partition is $(\lambda_{k+1},\ldots,\lambda_m)$ in both cases. Then, the same argument as for $i_1$ proves that the lemma holds for row $i_{k+1}$, and this ends the proof.
\end{proof}

Eventually, we can now provide a proof of Theorem~\ref{TheoVexiType}.

\begin{proof}[Proof~of~Theorem~\ref{TheoVexiType}]
\textbf{Point (1):} this is an immediate consequence of Theorem~\ref{TheoStanVexil} together with Proposition~\ref{LiEx}.

\textbf{Point (2):} first, note that by definition of the line-exchange algorithm, we have
\begin{equation}\label{EqExchangeIdemp}
\text{for all} \ \T\in \Typ(S), \ \text{we have} \ (\T^L)^L=\T^L \ \text{and} \ (\T^C)^C=\T^C.
\end{equation}

Let us denote by $\theta$ the valuation associated with $\T_{\sigma}^L$ and by $l_i$ the number of boxes in row $i$ of $\Sh(\T_{\sigma}^L)$. Assume by contradiction that there exists an integer $k$ such that $l_k<l_{k+1}$. Then, thanks to Lemma~\ref{VexCarComb} we have that $\Sh(\T_{\sigma}^L)$ is as represented on Figure~\ref{FigStep1}.
\begin{figure}[!h] 
\includegraphics[width=8 cm]{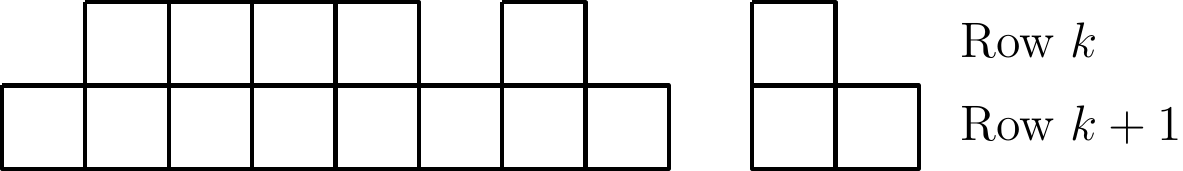}
\caption{Rows $k$ and $k+1$ of $\Sh(\T_{\sigma}^L)$.}\label{FigStep1}
\end{figure}
However, by construction for all $(a,k) \in \Sh(\T_{\sigma}^L)$ we have $\theta(a,k) > \theta(a+1,k)$. Therefore, if we perform the line-exchange algorithm on $\T_{\sigma}^L$, then these two rows are exchanged, so that we have 
\[
(\T_{\sigma}^L)^L \neq \T_{\sigma}^L,
\] 
contradicting \eqref{EqExchangeIdemp}. Thus, the sequence $(l_i)_i$ is non-increasing. Using a similar argument, we prove that the sequence $(c_i)_i$ is non-increasing, where $c_i$ is the number of boxes in column $i$ of $\Sh(\T_{\sigma}^E)=\Sh((\T_{\sigma}^L)^C)$.

Eventually, the same arguments as for the proof of Lemma~\ref{VexCarComb} prove that $\Sh(\T_{\sigma}^E)$ is a partition, which is necessarily equal to $\lambda(\sigma)'$. 
\end{proof}

\subsection{Link with balanced tableaux}\label{SecBalanced}

Let us now explain how the construction made in the previous sections can be use to provide an alternative proof of Theorem~\ref{TheoEGBal} (but not fundamentally different from the one in \cite{FGR}).
Let $\lambda$ be a partition of an integer $n$ that we identify with its Ferrers diagram. A box $\mathfrak{c}=(a,b)$ of $\lambda$ is called a \emph{corner} of $\lambda$ if and only if there is no boxes on the right and below $\mathfrak{c}$, \emph{i.e.} both $(a+1,b)$ and $(a,b+1)$ are not in $\lambda$.

Let $\mathfrak{c}=(a,b)$ be a corner of $\lambda$ such that $k=\lambda_a+\lambda'_b-1$ is maximal (such a corner is not necessarily unique). Then, we can place $\lambda$ in the staircase partition $\lambda_{k+1}$ as shown on Figure~\ref{FigFall1}.
\begin{figure}[!h] 
\includegraphics[width=9 cm]{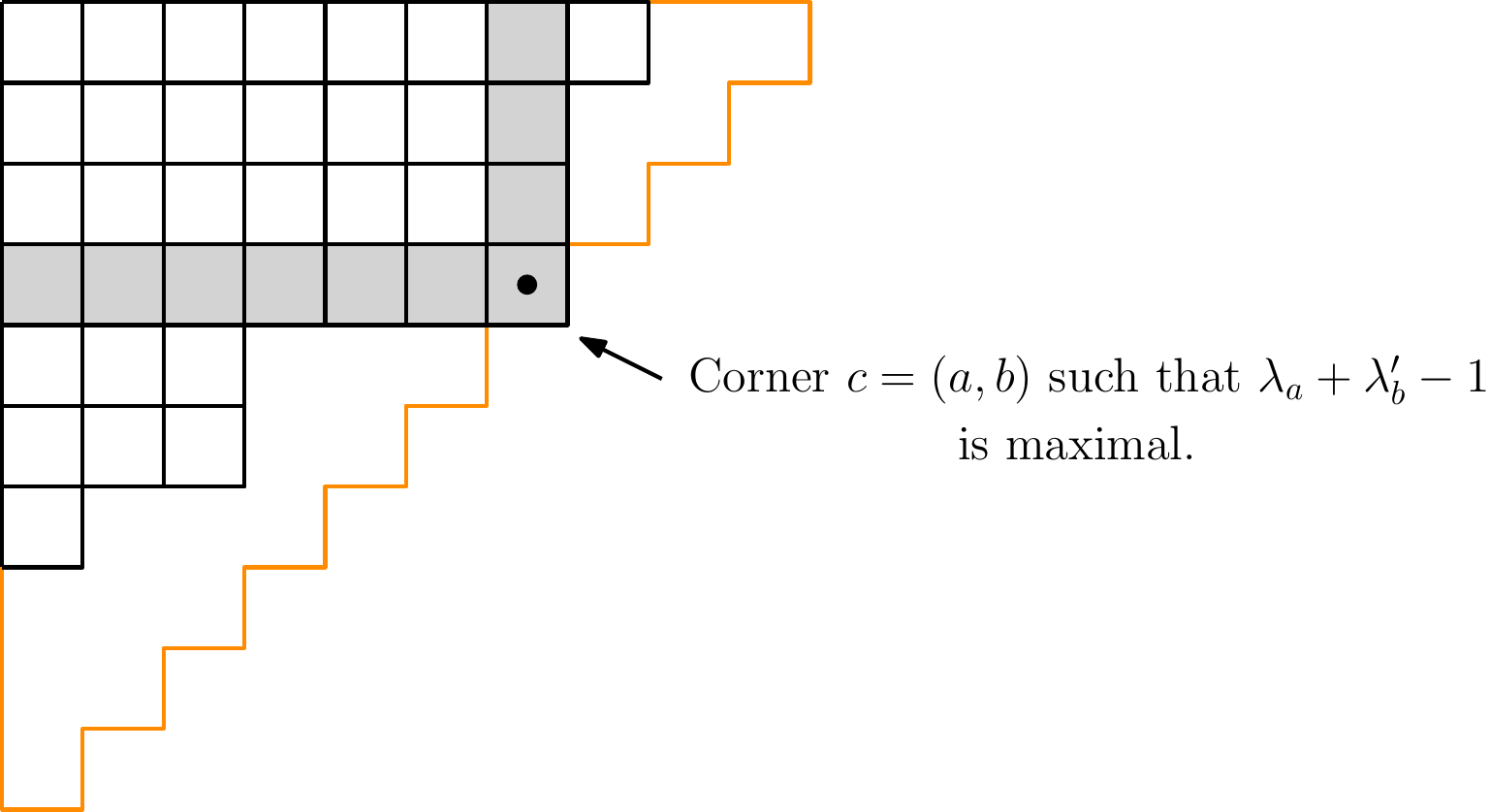}
\caption{The partition $(8,7,7,7,3,3,1)$ in $\lambda_{11}$.}\label{FigFall1}
\end{figure}

Let us look at the corners $(u,v)$ of $\lambda$ which are on the diagonal boundary of the staircase partition. For each such corner, we set $R_{(u,v)}=\{ (x,y) \in \lambda \ | \ x \leq u, \ y \leq v \}$ and we consider the union $R$ of the $R_{(u,v)}$. Then we let each connected component of $\lambda\setminus R$ fall in the staircase tableau as shown on Figure~\ref{Fall2}. We repeat the same procedure for each connected component of the resulting diagram, while it is possible. At the end, we get a sub-diagram of $\lambda_{k+1}$, which we denote by $S(\lambda)$.
\begin{figure}[!h] 
\includegraphics[width=11 cm]{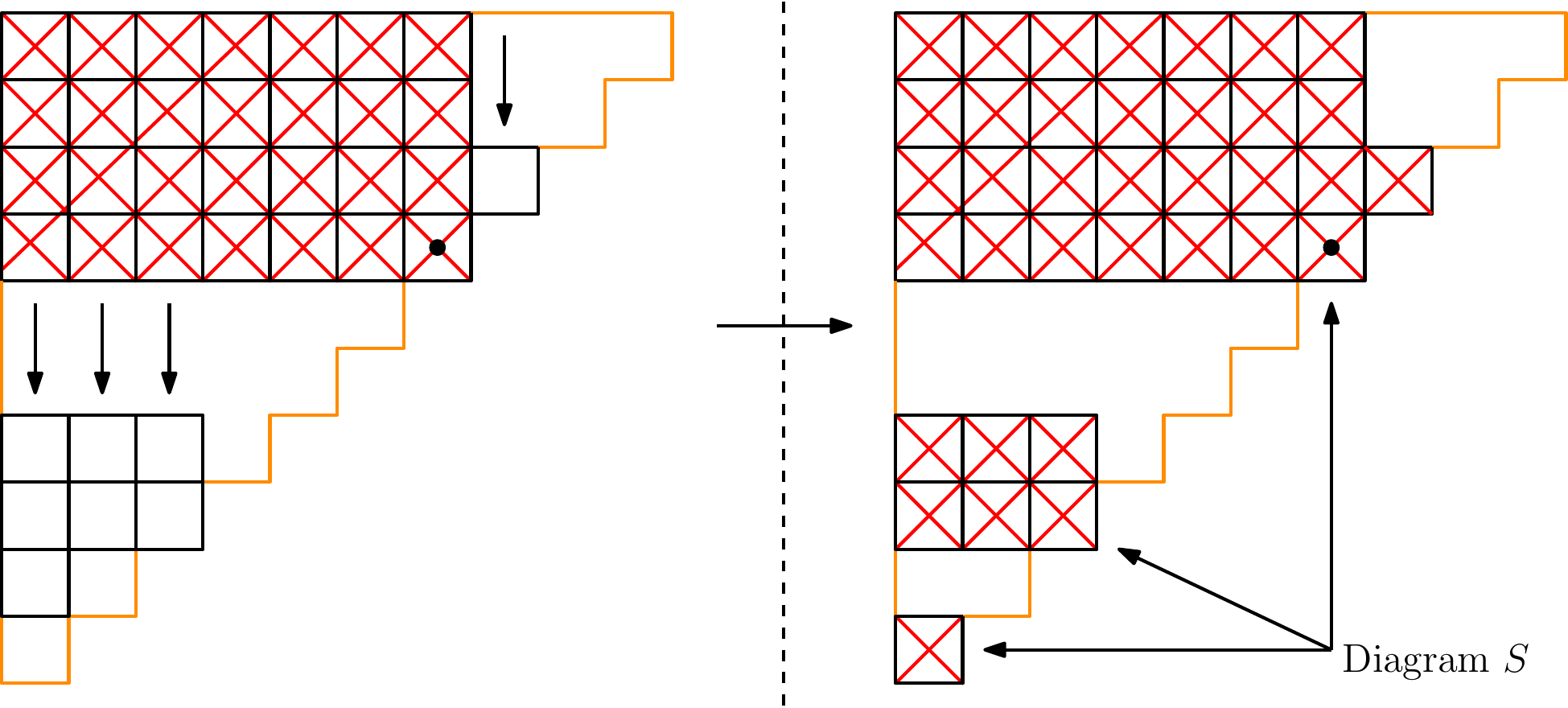}
\caption{}\label{Fall2}
\end{figure}

\begin{lemma}
There exists $\sigma_{\lambda} \in S_{k+1}$ such that $\Sh(\T_{\sigma(\lambda)})=S(\lambda)$. Moreover, $\sigma_{\lambda}$ is vexillary and $\lambda(\sigma)'=\lambda$.
\end{lemma}

\begin{proof}
The proof of this lemma requires the use of some results and notations from \cite{FV}, which we first recall. Let us consider the type $\A_{k+1}$ (see Section~\ref{SubSecTypPerm}) together with its associated valuation $\theta$. We consider the set $IS(\A_{k+1})$ made of all the sets $A \subset \lambda_{k+1}=\{ (a,b) \in \mathbb{N}^2 \ | \ 1 \leq a < b \leq k+1 \}$ such that:
\begin{enumerate}
\item for all $\mathfrak{c} \in A$, $\theta(\mathfrak{c}) \leq |A \cap \left( H_{(\lambda_{k+1})}(\mathfrak{c}) \setminus \{ \mathfrak{c} \} \right)|$;
\item for all $\mathfrak{c} \in \lambda_{k+1} \setminus A$, $\theta(\mathfrak{c}) \geq |A \cap \left( H_{(\lambda_{k+1})}(\mathfrak{c}) \setminus \{ \mathfrak{c} \} \right)|$.
\end{enumerate}
Thanks to \cite[Prop.~3.1]{FV} and \cite[Theorem~4.1]{FV}, a subset $A$ of $\lambda_{k+1}$ is the inversion set of a permutation in $S_{k+1}$ if and only if $A \in IS(\A_{k+1})$.

Our aim is now to prove that $S(\lambda) \in IS(\A_{k+1})$. Let $z=(x,y) \in \lambda_{k+1}$.
\begin{itemize}
\item If $z \notin S(\lambda)$, then for all $z' \in H_{\lambda_{k+1}}(z) \cap S(\lambda)$, $z'$ is in the same column and strictly below $z$. Moreover, by definition $\theta(z)$ equals the number of boxes strictly below $z$. Thus, we have $\theta(z) \geq |H_{\lambda_{k+1}}(z) \cap S(\lambda)|$;
\item If $z=(x,y) \in S(\lambda)$, then by construction of $S(\lambda)$ there exists $z'=(x',y')\in S(\lambda)$ such that:
$x' \geq x$, $y' \geq y$, $\theta(z')=0$, and 
\[
\text{for all} \ x \leq u \leq x' \ \text{and} \ y \leq v \leq y', \ \text{we have} \ (u,v) \in S(\lambda). 
\]
Thus, we have $\theta(z)=(x'-x)+(y'-y) \leq |H_{S(\lambda)}(z)|-1$.
\end{itemize}
Therefore, we have $S(\lambda) \in IS(\A_{k+1})$, so that there exists $\sigma_{\lambda}$ such that $\Sh(\T_{\sigma_{\lambda}})=S_{\lambda}$. Moreover, if we perform the stacking process on $S(\lambda)$, it is clear that both $XY$ and $YX$ processes end with the partition $\lambda(\sigma)$. Thus, $\sigma_{\lambda}$ is vexillary and this concludes the proof.
\end{proof}

\begin{proposition}
Let $\sigma_{\lambda}$ be the permutation whose inversion set is $S(\lambda)$. Then, we have
\[
\Tab(\T_{\sigma_{\lambda}}^L)=\Bal(\lambda).
\]
\end{proposition}

\begin{proof}
First, note that we have $\Sh(\T_{\sigma_{\lambda}}^L)=\lambda$ by construction of $S(\lambda)$. We denote by $\theta$ and $\theta'$ the valuations associated with the types $\T_{\sigma_{\lambda}}$ and $\T_{\sigma_{\lambda}}^L$, respectively. Let $\mathfrak{c}=(a,b) \in S(\lambda)$ such that all boxes in the same row and on the right of $\mathfrak{c}$ are not in $S(\lambda)$. Then, by construction of $S(\lambda)$ and by definition of $\theta$ we have:
\begin{itemize}
\item $\theta(\mathfrak{c})$ equals the number of indices $k<a$ such that row $k$ of $S(\lambda)$ contains more boxes than row $a$;
\item for all $\mathfrak{d}=(a,d) \in S(\lambda)$, we have $\theta(\mathfrak{d})=\theta(\mathfrak{c}) + d-b$.
\end{itemize}
Therefore, when we perform the line-exchange algorithm on $\T_{\sigma_{\lambda}}$ we have that the row $a$ of $\T_{\sigma_{\lambda}}$ is swapped with exactly $\theta(\mathfrak{c})$ rows below it. Thus, we have that for all $\mathfrak{d} \in \lambda$, 
\[
\theta'(\mathfrak{d})=a_{\lambda}(\mathfrak{d}),
\]
hence the elements of $\Tab(\T_{\sigma_{\lambda}})$ are the balanced tableaux of shape $\lambda$, and reciprocally. This ends the proof.
\end{proof}

The previous proposition, together with Proposition~\ref{LiEx}, immediately imply Theorem~\ref{TheoEGBal}.

\subsection{An equivalence relation between vexillary permutations}\label{SecEquivRelat}

At this point, a natural question arises: given two vexillary permutation $\sigma$ and $\omega$, when do we have $\T_{\sigma}^E=\T_{\omega}^E$ ? In this section we answer this section by exhibiting an equivalence relation $\sim_v$ on the set of vexillary permutations with the property that, for any two vexillary permutations $\sigma \in S_n$ and $\omega \in S_m$, $\T_{\sigma}^E=\T_{\omega}^E$ if and only if $\sigma \sim_v \omega$.

We first introduce a notation.
Let $\sigma \in S_n$, $p \leq n$ be the lowest integer such that $\sigma(p) \ne p$ and $q\leq n$ be the biggest integer such that $\sigma(q) \ne q$. We define 
$$\overline{\sigma} := [\sigma(p)-(p-1) \ ; \ \sigma(p+1)-(p-1) \ ; \ \ldots \ ; \ \sigma(q)-(p-1)].$$
Note that $\overline{\sigma}$ is an element of $S_{n+1-p-q}$ because of the choice of $p$ and $q$.

\begin{definition}
We say that $\sigma \sim_v \omega$ if and only if $\overline{\sigma} = \overline{\omega}$.
\end{definition}

\begin{Theorem}
Let $\sigma$ and $\omega$ be two vexillary permutations, then $\T_{\sigma}^E=\T_{\omega}^E$ if and only if $\sigma \sim_v \omega$.
\end{Theorem}

\begin{proof}
\textbf{Step 1:} we begin with giving a combinatorial interpretation for the relation $\sim_v$. Let $\sigma \in S_n$ and $\omega \in S_m$ such that $\sigma \sim_v \omega$. Without loss of generality, we can assume that $n$ is larger than $m$. Then, we can see $\omega$ as a permutation of $S_n$ by adding $(n-m)$ fixed points at the end of $\omega$. Let us denote by $p_{\sigma}$ (resp. $p_{\omega}$) the smallest integer such that $\sigma(p_{\sigma}) \neq p_{\sigma}$ (resp. $\omega(p_{\omega}) \neq p_{\omega}$). By definition, we have $\Inv (\overline{\sigma})=\Inv(\overline{\omega})$ and
\[ \Inv (\sigma)=\{ (x+(p_{\sigma}-1),y+(p_{\sigma}-1)) \ | \ (x,y)\in \Inv(\overline{\sigma}) \}.\]
Thus, if we look at $\T_{\sigma}$ and $\T_{\omega}$, we have the situation described on Figure~\ref{Figure19}. 
\begin{figure}[!h] 
\includegraphics[width=12 cm]{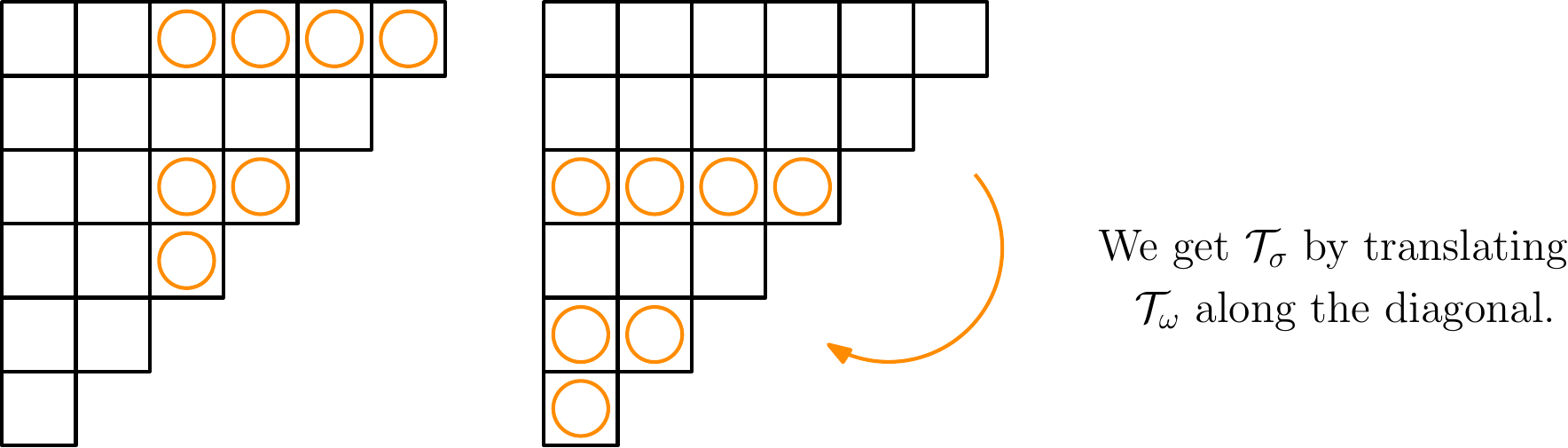}
\caption{The types $\T_{\sigma}$ and $\T_{\omega}$ seen as subsets of $\lambda_n$.}\label{Figure19}
\end{figure}
Then, we have $\T_{\sigma}^E=\T_{\omega}^E$ by construction.

\textbf{Step 2:} in order to prove the converse implication, we define two algorithms: one which reverses the Line-exchange Algorithm and another which reverse the Column-exchange Algorithm. Since these two algorithms are similar, we only give the definition of the algorithm on lines. Let $\T$ be a type of shape $S$, the \emph{reverse line-exchange algorithm} is the algorithm described below.
\begin{enumerate}
\item Erase all the empty rows of $\T$.
\item Set $i:=r$, where $r$ is the number of non-empty rows of $\T$.
\begin{enumerate}
\item If $i$ is a dethroned row of $\T$, then set $\T:=\T\! \! \uparrow_i$ and go back to step (2). Otherwise, go to step (2-b).
\item If there is no row above $i$, then the algorithm stops. Otherwise, set $i:=i-1$ and go back to step (2-a).
\end{enumerate}
\end{enumerate}
We denote by $^L\T$ the type obtained after we perform the reverse line-exchange algorithm. We also denote by $\overline{\T}$ the type obtained from $\T$ deleting empty rows.
In general, the type $^L(\T^L)$ is different from the type $\overline{\T}$. However, it is clear that $^L(\T^L)=\overline{\T}$ whenever there is no dethroned lines in $\overline{\T}$, and by construction there is no dethroned lines in $\overline{\T_{\sigma}}$ for any $\sigma \in S_n$. 

Therefore, for all vexillary permutations $\sigma,\omega \in S_n$, if $\T_{\sigma}^E=\T_{\omega}^E$ then we have $\overline{\T_{\sigma}}=\overline{\T_{\omega}}$, and this implies that $\T_{\sigma}=\T_{\omega}$.
Then, we have $\Inv(\overline{\sigma})=\Inv(\overline{\omega})$ implying that $\overline{\sigma}=\overline{\omega}$, \emph{i.e.} $\omega \sim_v \sigma$. This concludes the proof.
\end{proof}

\section{Combinatorial properties of types $\T_{\sigma}^E$}\label{SecCombinProp}

We finish this article by studying two combinatorial properties of types $\T_{\sigma}^E$, relative respectively to \emph{Schur functions} and partial fillings of tableaux.

\subsection{Semi-standard tableaux of type $\T_{\sigma}^E$ and Schur function}

As explained in \cite{FGR}, one can define a notion of ``semi-standard balanced tableaux" of a given shape $\lambda \vdash n$, thanks to the concept of \emph{column-strict balanced labelling} of a diagram, that we now detail by adapting it to our terminology. Let $\lambda$ be a partition and $\B$ be the type associated with $\Bal(\lambda)$ (see Example~\ref{ExDefBalStand}). Let $T=(t_{\mathfrak{c}})_{\mathfrak{c}\in \lambda}$ be a filling of $\lambda$ with integer, which is not necessarily a tableau (for instance, a given integer can appear several times in $T$). We say that $T$ is a column-strict balanced labelling of $\lambda$ if and only if there exists a filling sequence $L=[\mathfrak{c}_1,\ldots,\mathfrak{c}_n]$ such that:
\begin{itemize}
\item for all $i<j$, $t_{\mathfrak{c}_i} \leq t_{\mathfrak{c}_j}$;
\item for all $i<j$, if $\mathfrak{c}_i$ and $\mathfrak{c}_j$ are in the same column of $\lambda$, then $t_{\mathfrak{c}_i} < t_{\mathfrak{c}_j}$.
\end{itemize}
Let us denote by ${\rm SST}(\B)$ the set of all column-strict balanced labellings of $\lambda$ (the ${\rm SST}$ stands for ``\emph{semi-standard tableau of type $\B$}''). To each element $T=(t_{\mathfrak{c}})_{\mathfrak{c}\in \lambda}$ of ${\rm SST}(\B)$, we can associate a monomial 
$\displaystyle{x^T=\prod_{\mathfrak{c}\in \lambda}x_{t_{\mathfrak{c}}}}$, and thanks to \cite{FGR} we have:
\[
s_{\lambda}(x_1,x_2,\ldots)=\sum_{T \in {SST}(\B)}x^T,
\]
where $s_{\lambda}$ is the Schur function of shape $\lambda$.

All these definitions admit a straightforward generalization for all type, and we associate each type $\T$ with a set ${\rm SST}(\T)$ of semi-standard tableaux of type $\T$. We then have the following property.

\begin{proposition}
Let $\sigma$ be a vexillary permutation. Then, we have 
\[
s_{\lambda(\sigma)}=\sum_{T \in {SST}(\T_{\sigma}^E)}x^T.
\]
\end{proposition}

\begin{proof}
We have a one-to-one correspondence between the elements of ${\rm SST}(\T_{\sigma}^E)$ and the elements of $\SSF(\Inv(\sigma))$ introduced in \cite[Section~5]{FV}. More precisely, we go from one set to the other just by swapping column and lines according to the line and column exchange process. Therefore, the result follow immediately from \cite[Theorem~5.10]{FV}. 
\end{proof}

\subsection{Partial fillings of tableaux of type $\T_{\sigma}^E$}

It is classical that the number standard tableaux of a given shape $\lambda$ satisfy the following relation:
\begin{equation*}\label{EqSumCorner}
|{\rm SYT}(\lambda)|= \sum_{\lambda^-}|{\rm SYT}(\lambda^-)|,
\end{equation*}
where the sum ranges on the set of all the partition $\lambda^-$ obtained from $\lambda$ by suppressing a corner of its Ferrers diagram. Thanks to Theorem~\ref{TheoEGBal}, \eqref{EqSumCorner} also holds for balanced tableaux. In fact, proving \eqref{EqSumCorner} for balanced tableaux is clearly equivalent to proving Theorem~\ref{TheoEGBal}, and this approach naturally leads to study the number of balanced tableaux such that the integers $1,2,\ldots,k$ appear at given fixed positions. The study of this problem has been started in \cite{EG}, but was only solved in some specific cases. In this section, we give a complete solution to this problem in the general case of the type $\T_{\sigma}^E$ with $\sigma$ vexillary, and this can be applied to balanced tableaux as well. Before moving to the statement and the proof of our results, let us make two important remarks.
\begin{enumerate}
\item The results presented here are implicit in \cite{FGR} in the case of balanced tableaux, and can be deduced with the same arguments as the ones used here.
\item Our approach gives a more direct proof that \eqref{EqSumCorner} holds for balanced tableaux, as required in \cite[Section~2]{EG}. However, our current proof does not lead to a satisfactory combinatorial proof of \ref{TheoEGBal} since it still uses Theorem~\ref{TheoStanVexil} as fundamental compound. Nevertheless, the concepts developed in the current article might be a first step toward such a proof. 
\end{enumerate}

We begin with introducing a useful notation.

\begin{definition}
Let $\sigma\in S_m$ be a vexillary permutation such that $\lambda(\sigma) \vdash n$ and $U=[z_1,\ldots,z_{k}]$ be a sequence of boxes of the Ferrers diagram of $\lambda(\sigma)$. We denote by $N_{\sigma,U}$ the set defined by
\[
N_{\sigma,U}:=\{T=(t_{\mathfrak{c}})_{\mathfrak{c}\in \lambda} \in \Tab(\T_{\sigma}^E) \ | \ t_{z_i}=i \ \text{for all} \ i\in [k]  \}.
\]
\end{definition}

\begin{Theorem}\label{PartialFil}
Let $\sigma\in S_m$ be a vexillary permutation such that $\lambda(\sigma) \vdash n$ and $U=[z_1,\ldots,z_{k}]$ be a sequence of boxes of the Ferrers diagram of $\lambda(\sigma)$.
Then, we have that either $N_{\sigma,U}$ is empty, or there exists $\omega\in S_m$ such that
\[
|N_{\sigma,U}|=|\Red(\omega)|.
\]
\end{Theorem}

\begin{proof}
Let us assume that $N_{\sigma,U}$ is not empty.
By construction, there exists a bijection $\Psi$ between the set of boxes of $\T_{\sigma}$ and the inversion set of $\sigma$ such that for any sequence $L=[x_1,\ldots,x_n]$, we have that $L $ is a filling sequence of $\T_{\sigma}^E$ if and only if 
$[\Psi(x_1),\Psi(x_2),\ldots,\Psi(x_n)] \ \text{is a filling sequence of} \ \T_{\sigma}.$
Consequently, thanks to \cite[Theorem~4.1]{FV}, $\Psi(\{z_1,\ldots,z_k\})$ is the inversion set of a permutation $\tau \in S_m$, and we have a one-to-one correspondence between $N_{\sigma,U}$ and the set of the maximal chains from $\tau$ to $\sigma$ in $(S_m,\leq_R)$. Thanks to \cite[Prop. 3.1.6, p. 69]{BB}, the set of maximal chains from $\tau$ to $\sigma$ in $(S_m,\leq_R)$ is in one-to-one correspondence with $\Red(\tau^{-1}\omega)$, and this concludes the proof.
\end{proof}

This theorem applies in particular to balanced tableaux. Moreover, this result is constructive: the permutation $\omega$ can be computed. We cannot provide a systematic description of the associate permutation, however we have the following combinatorial result.

\begin{Theorem}\label{NicePartial}
Let $\sigma\in S_m$ be a vexillary permutation such that $\lambda(\sigma) \vdash n$ and $U=[z_1,\ldots,z_{k}]$ be a sequence of boxes of the Ferrers diagram of $\lambda$. If there exists a partition $\mu\vdash (n-k)$ such that the resulting partition of the $XY$ and $YX$ processes applied on the diagram $\lambda(\sigma) \setminus \{z_1,\ldots,z_k\}$ is $\mu$, then $|N_{\sigma,U}|=f^{\mu}$. 
\end{Theorem}

\begin{proof}
We keep the notations introduced in the proof of Theorem~\ref{PartialFil}. Let $\tau$ be the permutation whose inversion set is given by $\{ \Psi(x_1),\ldots,\Psi(x_k)\}$. Since $\omega = \tau^{-1} \sigma$ with $\tau \leq_R \sigma$, we have 
\[
\Inv(\omega)=\tau^{-1}(\Inv(\sigma) \setminus \Inv(\tau))=\tau^{-1}(\Sh(\T_{\sigma}) \setminus \{ \Psi(x_1),\ldots,\Psi(x_k)\}).
\]
Thus, if we denote by $l_i$ (resp. $c_i$) the number of boxes in row (resp. column) $i$ of $\lambda(\sigma) \setminus \{z_1,\ldots,z_k\}$, we have that the sequences $(l_i)_i$ and $(g_i(\omega))_i$ (resp. $(c_i)_i$ and $(r_i(\omega))_i$) are equal up to re-ordering. Therefore, $\omega$ is vexillary and $\mu(\omega)=\mu$. This concludes the proof.
\end{proof}

Notice that the results of the current section apply to balanced tableaux.

\begin{definition}
Let $\lambda=(\lambda_1,\ldots,\lambda_k)$ be a partition of $n$ and $(a,b),(c,d) \in \lambda$, we say that $(a,b)$ and $(c,d)$ are in the same \emph{block} if and only if $\lambda_a=\lambda_c$. 
Let $B$ be a block of $\lambda$ and let $i$ be the minimal integer such that $(i,\lambda_i) \in B$, then the box $(i,\lambda_i)$ is called the \emph{upper right corner} of $B$.
\end{definition}

Let $T$ be a balanced tableau of shape $\lambda \vdash n$. By definition, we have that the integer $1$ appears in the upper right corner of a block (see Figure~\ref{EndIllu}).
\begin{figure}[!h] 
\includegraphics[width=8.5 cm]{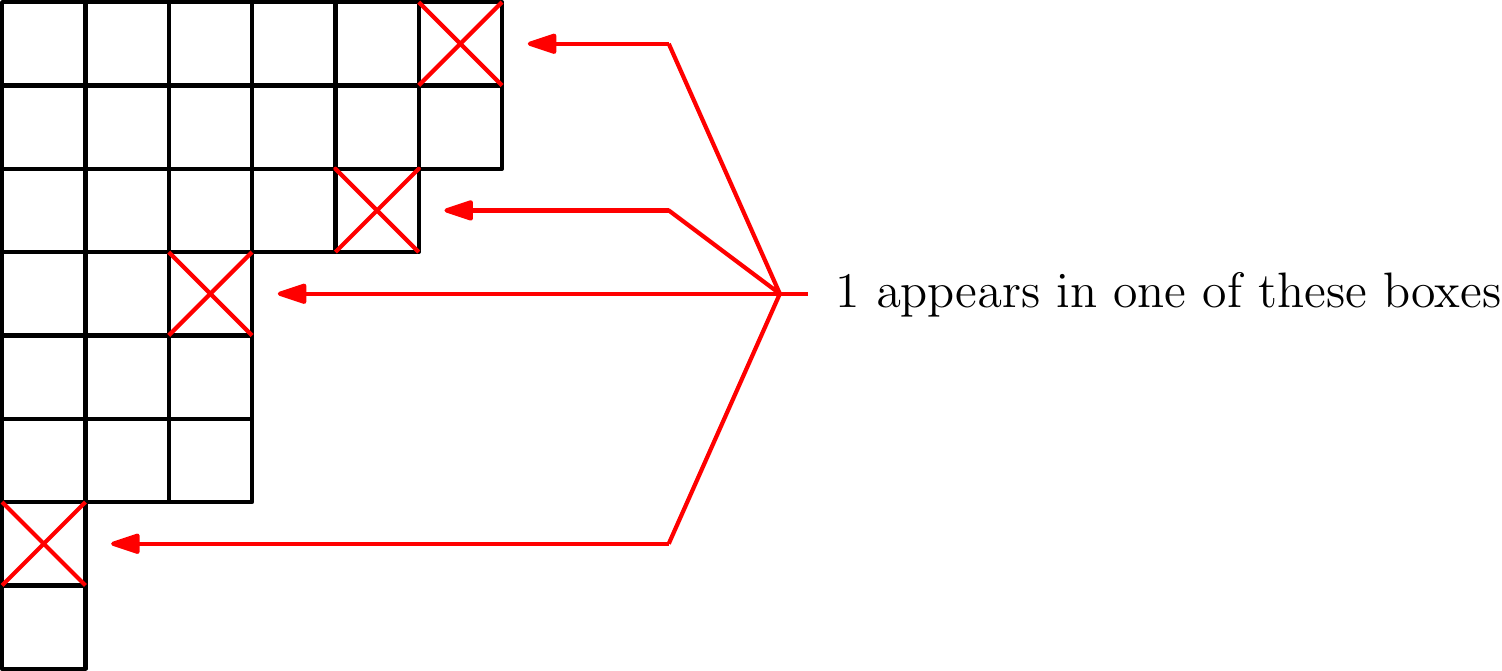}
\caption{Possible positions of the integer 1 in a ballanced tableau.}\label{EndIllu}
\end{figure}
Then, thanks to Theorem~\ref{NicePartial} we immediately have the following result.

\begin{proposition}
Let $B$ be a block of $\lambda \vdash n$ and $\mathfrak{c}$ be the upper right corner of $B$. Then, the number of balanced tableaux of shape $\lambda$ such that $1$ is in the box $\mathfrak{c}$ equals the number of standard tableaux of shape $\lambda^-$, where $\lambda^- \vdash n-1$ is obtained from $\lambda$ by suppressing the corner of the block $B$. 
\end{proposition}

The previous proposition directly implies that for all partition $\lambda$,
\[
|\Bal(\lambda)|=\sum_{\lambda^-} |\Bal(\lambda^-)|,
\]
where the sum range on all partitions obtained from $\lambda$ by suppressing a corner.


\begin{thebibliography}{SGA}

\bibitem[1]{Bjor} A. Bjorner, \textit{Orderings of Coxeter groups}. Contemp. Math., 34, Amer. Math. Soc., Providence, RI, 1984.

\bibitem[2]{BB} A. Bjorner, F. Brenti, \textit{Combinatorics of Coxeter groups}. Springer, New York, NY, 2005.

\bibitem[3]{EG} P. Edelman, C. Greene, \textit{Balanced tableaux}. Adv. Math., 63, 42-99, 1987.

\bibitem[4]{FGR} S. Fomin, C. Greene, V. Reiner, M. Shimozono, \textit{Balanced Labellings and Schubert Polynomials}. Europ. J. Combin., 18, 373-389, 1995.

\bibitem[5]{HookProb} C. Greene, A. Nijenhuis, H. S. Wilf, \textit{A probabilistic proof of a formula for the number of Young Tableaux of a given shape}. Adv. in Math. 31 (1979), no. 1, 104-109.

\bibitem[6]{Hump} J.E. Humphreys, \textit{Reflection groups and Coxeter groups}. Cambridge Studies in advanced Mathematics, vol. 29, Cambridge University Press, Cambridge, 1990.

\bibitem[7]{LS} A. Lascoux, M.-P. Sch\"utzenberger, \textit{Schubert polynomials and the Littlewood-Richardson rule}. Lett. Math. Phys. 10 (1985), no. 2-3, 111-124.

\bibitem[8]{Little} D. Little, \textit{Combinatorial aspects of the Lascoux-Schützenberger tree}. Adv. Math. 174 (2003), no. 2, 236-253.

\bibitem[9]{NPS} J.-C. Novelli, I. Pak, A. V. Stoyanovskii, \textit{A direct bijective proof of the hook-length formula}. Discrete Math. Theor. Comput. Sci. 1 (1997), no. 1, 53-67.

\bibitem[10]{S1} R. Stanley, \textit{On the number of reduced decompositions of elements of Coxeter groups}. Europ. J. Combin., 5, 359-372, 1984.

\bibitem[11]{S2} R. Stanley, \textit{Enumerative combinatorics. Vol. 1}, volume 62 of \textit{Cambridge studies in advanced mathematics}. Cambridge university press, Cambridge, 1999. With a foreword by Gian-Carlo Rota and appendix 1 by Sergey Fomin.

\bibitem[12]{FV} F. Viard, \textit{A new family of posets generalizing the weak order on some Coxeter groups}. Preprint arXiv.org:1508.06141, 2015 (39 pages).

\bibitem[13]{FV3} F. Viard, \textit{How to get the weak order out of a digraph}. DMTCS Proceedings, north America, FPSAC 2015. Available at: http://fpsac2015.sciencesconf.org/71019/document.

\bibitem[14]{FVThesis} F. Viard, \textit{From valued digraphs to complete lattices: a new approach of the weak order on Coxeter groups}. PhD Thesis, Universit\'e Claude Bernard. Available at http://math.univ-lyon1.fr/homes-www/viard/These.pdf

\end{thebibliography}
\end{document}